\documentclass[11pt, reqno]{amsart}
\usepackage{geometry}
\usepackage{graphicx}
\usepackage{amssymb}
\usepackage{amsmath, mathrsfs, stmaryrd}
\usepackage{color}
\newtheorem{theorem}{Theorem}
\newtheorem{definition}[theorem]{Definition}
\newtheorem{lemma}[theorem]{Lemma}
\newtheorem{corollary}[theorem]{Corollary}
\newtheorem{proposition}[theorem]{Proposition}
\newtheorem{remark}[theorem]{Remark}

\newcommand{\R}{\mathbb R}
\numberwithin{theorem}{section}
\numberwithin{equation}{section}
\title[Quasi-local energy with respect to de Sitter/anti-de Sitter reference]{Quasi-local energy with respect to \\ de Sitter/anti-de Sitter reference}
\author{Po-Ning Chen, Mu-Tao Wang, and Shing-Tung Yau}
\date{\today}
\thanks{P.-N. Chen is supported by NSF grant DMS-1308164, M.-T. Wang is supported by NSF grants DMS-1105483 and DMS-1405152,  and S.-T. Yau is supported by NSF
grants  PHY-0714648 and DMS-1308244. This work was partially supported by a grant from the Simons Foundation (\#305519 to Mu-Tao Wang). Part of this work was carried out
when P.-N. Chen and M.-T. Wang were visiting the Department of Mathematics and the Center of Mathematical Sciences and Applications at Harvard  University.} 

\begin{document}
\maketitle
\begin{abstract}
This article considers the quasi-local conserved quantities with respect to a reference spacetime with a cosmological constant. We follow the approach developed by the authors in \cite{Wang-Yau1,Wang-Yau2,Chen-Wang-Yau3} and define the quasi-local energy as differences of  surface Hamiltonians. The ground state for the gravitational energy is taken to be a reference configuration in  the de Sitter (dS) or Anti-de Sitter (AdS) spacetime. This defines the quasi-local energy with respect to the reference spacetime and generalizes our previous definition with respect to the Minkowski spacetime.
Through an optimal isometric embedding into the reference spacetime, the Killing fields of the reference spacetime are transplanted back to the surface in the physical spacetime to complete the definitions of quasi-local conserved quantities.
We  also compute how the corresponding total conserved quantities evolve under the Einstein equation with a  cosmological constant.
\end{abstract}

\section{Introduction} In \cite{Wang-Yau1, Wang-Yau2, Chen-Wang-Yau3}, the authors developed the theory of quasi-local energy (mass) and quasi-local conserved quantities in general relativity with respect to the Minkowski spacetime reference. In view of recent astronomical observations, the current article embarks on the study of the corresponding theory with respect to a reference spacetime with a non-zero cosmological constant. In particular, the quasi-local energy and quasi-local conserved quantities with respect to the dS or AdS spacetime are defined in this article. The construction, similar to the Minkowski
reference case, is based on the Hamilton-Jacobi analysis of the gravitational action and optimal isometric embeddings.
However, the result, not only is more complicated, but also reveals new phenomenon due to the nonlinear nature of the reference spacetime. The construction employs ideas developed by the authors in \cite{Wang-Yau07} (see also \cite{Shi-Tam}) for quasi-local mass with respect to the hyperbolic reference.

In the following, we review the definition of the quasi-local energy-momentum in  \cite{Wang-Yau1,Wang-Yau2} with respect to the Minkowski spacetime. The main motivation of this definition is the rigidity property that surfaces in the Minkowski spacetime should have zero mass. As a result, all possible isometric embeddings $X$ of the surface into $\R^{3,1}$ are considered and an energy is assigned to each pair $(X, T_0)$ of an isometric embedding $X$ and a constant future timelike unit vector field $T_0$  in $\R^{3,1}$.

Let $\Sigma$ be a closed embedded spacelike 2-surface in a spacetime. We assume the mean curvature vector $H$ of $\Sigma$  is spacelike and the normal bundle of $\Sigma$ is oriented. The data used in the definition of the Wang-Yau quasi-local mass is the triple $(\sigma,|H|,\alpha_H)$, in which $\sigma$ is the induced metric on $\Sigma$, $|H|$ is the norm of the mean curvature vector and $\alpha_H$ is the connection one-form of the normal bundle with respect to the mean curvature vector
\[ \alpha_H(\cdot )=\langle \nabla^N_{(\cdot)}   \frac{J}{|H|}, \frac{H}{|H|}  \rangle  \]
where $J$ is the reflection of $H$ through the incoming  light cone in the normal bundle.

Given an isometric embedding $X:\Sigma\rightarrow \R^{3,1}$ and a constant future timelike unit vector field $T_0$  in $\R^{3,1}$, let $\widehat{X}$ be the projection of ${X}$ onto the orthogonal complement of $T_0$. We denote the induced metric, the second fundamental form, and the mean curvature of the image by $\hat{\sigma}_{ab}$, $\hat{h}_{ab}$, and $\widehat{H}$, respectively.  The Wang-Yau quasi-local energy with respect to $(X, T_0)$ is 
\[\begin{split}&E(\Sigma, X, T_0)=\frac{1}{8\pi} \int_{\widehat{\Sigma}}\widehat{H}d{\widehat{\Sigma}}-\frac{1}{8\pi}\int_\Sigma \left[\sqrt{1+|\nabla\tau|^2}\cosh\theta|{H}|-\nabla\tau\cdot \nabla \theta -\alpha_H ( \nabla \tau) \right]d\Sigma,\end{split}\] where $\theta=\sinh^{-1}(\frac{-\Delta\tau}{|H|\sqrt{1+|\nabla\tau|^2}})$, $\nabla$ and $\Delta$ are the gradient and Laplacian, respectively, with respect to $\sigma$ and $\tau= - \langle X ,T_0 \rangle$ is the time function.

In \cite{Wang-Yau1,Wang-Yau2}, it is proved that, $E(\Sigma, X, T_0) \ge 0$ if $\Sigma$ bounds a spacelike hypersurface in $N$, the dominant energy condition holds in $N$, and the pair $(X,T_0)$ is admissible. The Wang-Yau quasi-local mass is defined to be the minimum  of the quasi-local energy $E(\Sigma, X, T_0)$ among all admissible pairs $(X, T_0)$. In particular, for a surface in the Minkowski spacetime, its Wang-Yau mass is zero. However, for surfaces  in a general spacetime, it is not clear which isometric embedding would minimize the quasi-local energy. To find the isometric embedding that minimizes the quasi-local energy, we study the Euler-Lagrange equation for the critical point of the Wang-Yau energy. It is the following fourth order nonlinear elliptic equation (as an equation for $\tau$)

\begin{equation} \label{optimal}
 -(\widehat{H}\hat{\sigma}^{ab} -\hat{\sigma}^{ac} \hat{\sigma}^{bd} \hat{h}_{cd})\frac{\nabla_b\nabla_a \tau}{\sqrt{1+|\nabla\tau|^2}}+ div_\sigma (\frac{\nabla\tau}{\sqrt{1+|\nabla\tau|^2}} \cosh\theta|{H}|-\nabla\theta-\alpha_{H})=0
\end{equation}
coupled with the isometric embedding equation for $X$. \eqref{optimal} is referred to as the optimal isometric embedding equation.

The data for the image surface of the isometric embedding $X$ in the Minkowski spacetime can be used to simplify the expressions for the quasi-local energy and the optimal isometric embedding equation. Denote the norm of the mean curvature vector and the connection one-form in mean curvature gauge of $X(\Sigma)$ in $\R^{3,1}$ by $|H_0|$ and $\alpha_{H_0}$, respectively. Let $\theta_0=\sinh^{-1}(\frac{-\Delta\tau}{|H_0|\sqrt{1+|\nabla\tau|^2}})$. We have the following identities relating the geometry of the image of the isometric embedding $X$ and 
the image surface $\widehat{\Sigma}$ of $\widehat{X}$  \cite{Chen-Wang-Yau2}.
\[\sqrt{1+|\nabla\tau|^2}\widehat{H} =\sqrt{1+|\nabla\tau|^2}\cosh\theta_0|{H_0}|-\nabla\tau\cdot \nabla \theta_0 -\alpha_{H_0} ( \nabla \tau) \]
\[   -(\widehat{H}\hat{\sigma}^{ab} -\hat{\sigma}^{ac} \hat{\sigma}^{bd} \hat{h}_{cd})\frac{\nabla_b\nabla_a \tau}{\sqrt{1+|\nabla\tau|^2}}+ div_\sigma (\frac{\nabla\tau}{\sqrt{1+|\nabla\tau|^2}} \cosh\theta_0|{H_0}|-\nabla\theta_0-\alpha_{H_0})=0.
\]
The second identity states that a surface inside $\R^{3,1}$ is a critical point of the quasi-local energy with respect to other isometric embeddings back to $\R^{3,1}$. This can be proved by either the positivity of the quasi-local energy or a direct computation. We  substitute these relations into the expression for $E(\Sigma, X, T_0)$ and the optimal isometric embedding equation, and rewrite them in terms of a function $\rho$ and a one-form $j$ with

\[
\begin{split}
\rho &= \frac{\sqrt{|H_0|^2 +\frac{(\Delta \tau)^2}{1+ |\nabla \tau|^2}} - \sqrt{|H|^2 +\frac{(\Delta \tau)^2}{1+ |\nabla \tau|^2}} }{ \sqrt{1+ |\nabla \tau|^2}},
\\ 
j_a & =\rho {\nabla_a \tau }- \nabla_a [ \sinh^{-1} (\frac{\rho\Delta \tau }{|H_0||H|})]-(\alpha_{H_0})_a + (\alpha_{H})_a.\end{split}
\]

In terms of these, the quasi-local energy is 
\[E(\Sigma,X,T_0)=\frac{1}{8\pi}\int_\Sigma (\rho+j_a\nabla^a\tau)\]
and a pair $(X, T_0)$  of an embedding $X:\Sigma\hookrightarrow \mathbb{R}^{3,1}$ and an observer $T_0$ satisfies the optimal isometric embedding equation \eqref{optimal} if $X$ is an isometric embedding and 
\[
div_{\sigma} j=0.
\] 

In \cite{Chen-Wang-Yau3}, the quasi-local conserved quantity of $\Sigma$ with respect to a pair $(X, T_0)$ of optimal isometric embedding and a Killing field $K$ is defined to be
\[
E(\Sigma, X, T_0, K)=-\frac{1}{8\pi} \int_\Sigma
\left[ \langle K, T_0\rangle \rho+j(K^\top) \right]d\Sigma
\]
where $K^\top$ is the tangential part of $K$ to $X(\Sigma)$.

The article is organized as follows: in Section 2, we gather results for the geometry of surfaces in the reference spacetime (dS or AdS). In Section 3, we derive a conservation law for surfaces in the reference spacetime. The conservation law is used in Section 4 to define the quasi-local energy. In Section 5, the first variation of the quasi-local energy is derived. In Section 6, the second variation of the quasi-local energy is computed and we prove that a surface in the static slice of the reference spacetime is a local minimum of its own quasi-local energy. In Section 7, we define the quasi-local conserved quantities and evaluate their limits for an asymptotically AdS initial data, and compute how these conserved quantities evolve under the Einstein equation with a  cosmological constant. 
\section{Geometry of surfaces in the reference spacetime }
In this section, we gather results for the geometry of surfaces in the reference spacetime, which refers to the dS spacetime or AdS spacetime throughout the article. In a static chart $(t, x^1, x^2, x^3)$ of the reference spacetime, the metric is of the form
\begin{equation}\label{metric_form}
\check{g}= -\Omega^2 dt^2 +g_{ij} dx^i dx^j
\end{equation}
where $g_{ij}$ is the hyperbolic metric for the AdS spacetime, or the round metric on $S^3$ for the dS spacetime, and  $\Omega$ is the corresponding static potential. The metric is normalized such that the scalar curvature of $g_{ij}$ is $6 \kappa$ where $\kappa $ is  $1$ or $-1$. Denote the covariant derivative of the static slice by $\bar \nabla$ and that of the reference spacetime by $D$.

The static equation reads:
\[   (-\bar \Delta \Omega) g_{ij} + \bar \nabla_i \bar \nabla_j \Omega  -\Omega Ric_{ij}=0\]
where $Ric_{ij}$ is the Ricci curvature of the metric $g_{ij}$. In our case, a static slice is a space form and $g_{ij}$ and $\Omega$ satisfy
\begin{equation}  \label{static_potential_hessian}
\bar \nabla^2 \Omega  = -\kappa \Omega g. \end{equation}

Consider a surface $\Sigma$ in  the reference spacetime defined by an embedding $X$ from an abstract surface $\Sigma_0$ into the reference spacetime. In the static chart, we denote the components of $X$ by $(\tau,X^1,X^2,X^3)$ and refer to $\tau$ as the time function. Let $\sigma$ be the induced metric on $\Sigma$,  $H_0$ be the mean curvature vector of $\Sigma$  and $J_0$ be the reflection of $H_0$ through the light cone in the normal bundle of $\Sigma$. Denote the covariant derivative with respect to the induced metric $\sigma$ by $\nabla$. 

Given an orthonormal frame $\{ e_3,e_4\}$ of the normal bundle of $\Sigma$ in the reference spacetime where $e_3$ is spacelike and $e_4$ is future timelike, we define the connection one-form associated to the frame 
\begin{equation}\label{alpha_3}  \alpha_{e_3} (\cdot)=\langle D_{(\cdot)}  e_3, e_4 \rangle.\end{equation}
We assume the mean curvature vector of $\Sigma$ is spacelike and consider the following connection one-form of $\Sigma$ in the mean curvature gauge
\begin{equation}\label{alpha_h} \alpha_{H_0}(\cdot )=\langle D_{(\cdot)}   \frac{J_0}{|H_0|}, \frac{H_0}{|H_0|}   \rangle,  \end{equation}
where $J_0$ is the reflection of $H_0$ through the incoming  light cone in the normal bundle.

Let $\widehat \Sigma$ be the surface in the static slice $t=0$ given by $\widehat X=(0,X^1,X^2,X^3)$ which is assumed to be an embedding. The surfaces $\Sigma$ and $\widehat\Sigma$ are canonically diffeomorphic. Let $\hat \sigma$ be the induced metric on $\widehat \Sigma$, and $\widehat H$  and $\hat h_{ab}$ be the mean curvature and second fundamental form of $\widehat \Sigma$ in the static slice, respectively. Denote the covariant derivative with respect to the metric $\hat \sigma$ by $\hat \nabla$. 

The Killing vector field $\frac{\partial}{\partial t} $ generates a one-parameter family of isometries $\phi_t$ of the reference spacetime and we have from the form of the metric \eqref{metric_form}
\begin{equation}\label{christoffel} D_{\frac{\partial}{\partial t}} \frac{\partial}{\partial x^i}=\frac{\partial (\log \Omega)}{\partial x^i}\frac{\partial}{\partial t}.\end{equation}
Let $C$ be the image of $\Sigma$ under the one-parameter family $\phi_t$. The intersection of $C$ with the static slice $t=0$ is $\widehat \Sigma$. By a slight abuse of terminology, we refer to $\widehat \Sigma$ as the projection of $\Sigma$. We consider the following two vector fields on $C$. Let $v^a$ be a coordinate system on $\Sigma_0$ and consider the pushforward, $\widehat X_*(\frac{\partial}{\partial v^a})$, of $\frac{\partial}{\partial v^a}$ to $\widehat \Sigma$ by the embedding $\widehat X$. The pushforward of $\widehat X_*(\frac{\partial}{\partial v^a})$ to $C$ by the one-parameter family $\phi_t$ gives a vector field, still denoted by $\widehat X_*(\frac{\partial}{\partial v^a})$ on $C$. $\widehat X_*(\frac{\partial}{\partial v^a})$ is perpendicular to $\frac{\partial}{\partial t}$ everywhere on $C$. Similarly, we consider the pushforward of $X_*(\frac{\partial}{\partial v^a})$ to $C$ by the one-parameter family $\phi_t$.

The function $\tau$ can be viewed as a function on $\Sigma_0$ as well. $\tau_a=\frac{\partial \tau}{\partial v^a}$ is a one-form that lives on $\Sigma_0$, as well as $\Sigma$ and $\widehat \Sigma$, through the canonical diffeomorphism. 

 As tangent vector fields on $C$, we have
\begin{equation} X_*(\frac{\partial}{\partial v^a})=\widehat X_*(\frac{\partial}{\partial v^a}) + \tau_a  \frac {\partial}{\partial t}.  \end{equation}

Finally, let $\breve e_3$ be the outward unit normal of $\widehat \Sigma$ in the static slice $t=0$. Consider the pushforward of $\breve e_3$ by the one-parameter family  $\phi_t$, which is denoted by $\breve e_3$ again. Let $\breve e_4$ be the future directed unit normal of $\Sigma$ normal to $\breve e_3$ and extend it along $C$ in the same manner.  In particular, $X_*(\frac{\partial}{\partial v^a})$ is perpendicular to $\breve e_3$ and $\breve e_4$, and $\widehat X_*(\frac{\partial}{\partial v^a})$ is perpendicular to $\breve e_3$ and $\frac{\partial}{\partial t}$. 

We derive the formulae for comparing various geometric quantities on $\Sigma$ and $\widehat \Sigma$ in the remaining part of this section. Denote $\nabla\tau=\sigma^{ab}\tau_a\frac{\partial}{\partial v^b}$ and $\hat{\nabla}\tau= \hat \sigma^{ab}\tau_a\frac{\partial}{\partial v^b}$, which are identified with the corresponding tangent vector fields on $\Sigma$ and $\widehat{\Sigma}$, respectively.

We consider  $\sigma$ and $\hat \sigma$ as two Riemannian metrics on $\Sigma_0$, which are related as follows:
\begin{align}
\label{relation_lower_metric}\sigma_{ab} =& \hat \sigma_{ab} - \Omega^2 \tau_a \tau_b\\
\label{relation_upper_metric} \sigma^{ab} =& \hat \sigma^{ab} +\frac{ \Omega^2 \hat \nabla^a \tau  \hat \nabla^b  \tau}{1-\Omega^2|\hat \nabla \tau|^2}.
\end{align}
On $\Sigma_0$, $\nabla\tau$ and $\hat\nabla\tau$ are related as follows:
\begin{equation}
\label{relation_gradient}\nabla ^a \tau = \frac{\hat \nabla^a \tau}{1-\Omega^2|\hat \nabla \tau|^2}.
\end{equation}
This follows from a direct computation using equation \eqref{relation_lower_metric} and \eqref{relation_upper_metric}.

From \eqref{relation_gradient}, we derive
\begin{equation}\label{relation_area}
(1-\Omega^2|\hat \nabla \tau|^2)(1+\Omega^2| \nabla \tau|^2)=1.
\end{equation}

As before, we can extend $\nabla \tau$ and $\hat\nabla\tau$ along $C$. Along $C$,
\begin{equation}\label{grad_tau_C}\nabla\tau=(\nabla^a \tau) X_*(\frac{\partial}{\partial v^a}) \text{  and  } \hat \nabla\tau=(\hat \nabla^a \tau) \hat X_*(\frac{\partial}{\partial v^a}).\end{equation}
Note that along $C$, $\nabla \tau$ is perpendicular to $\breve{e}_3$ and $\breve{e}_4$ and $\hat\nabla\tau$ is perpendicular to $\breve{e}_3$ and $\frac{\partial}{\partial t}$. 
The following lemma expresses $\breve e_4$ and $ \frac {\partial}{\partial t}$ along $C$ in terms of each other.

\begin{lemma}
Along $C$, we have
\begin{align}\label{decompose_4} \breve e_4  = & \sqrt{1+\Omega^2| \nabla \tau|^2} \left(  \frac{\frac {\partial}{\partial t}}{\Omega} +  \Omega \hat \nabla \tau\right)\\
\label{decompose_t}  \frac {\partial}{\partial t} =& \Omega \sqrt{1+\Omega^2| \nabla \tau|^2} \breve e_4 - \Omega^2 \nabla \tau,
\end{align} where $\nabla\tau$ and $\hat\nabla\tau$ are given in \eqref{grad_tau_C}.
\end{lemma}
\begin{proof}
We first prove \eqref{decompose_4}. It is easy to see that $\frac {\partial}{\partial t} + \Omega^2 \hat \nabla \tau$ is normal to both $X_*(\frac{\partial}{\partial v^a})=\widehat X_*(\frac{\partial}{\partial v^a}) + \tau_a  \frac {\partial}{\partial t}$ and $\breve e_3$, and thus in the direction of $\breve e_4$. Moreover, its length is 
\[
\begin{split}
   \sqrt{\Omega^2  - \Omega^4 |\hat \nabla \tau|^2}
=&  \Omega \sqrt {1-\Omega^2|\hat \nabla \tau|^2 }= \frac{ \Omega}{\sqrt{1+\Omega^2| \nabla \tau|^2}}
\end{split}
\]
where  \eqref{relation_area} is used in the last equality. This proves  \eqref{decompose_4}. From \eqref{decompose_4}, we derive $\langle \frac{\partial}{\partial t}, \breve e_4\rangle=-\Omega\sqrt{1+\Omega^2 |\nabla\tau|^2}$.
This together with $\langle \frac{\partial}{\partial t}, X_*(\frac{\partial}{\partial v^b})\rangle=-\tau_b \Omega^2$ implies \eqref{decompose_t}.

\end{proof}
In the following proposition, we derive a formula relating the mean curvature $\widehat H$ of $\widehat \Sigma$ to geometric quantities on $\Sigma$. All geometric quantities on $\Sigma$ and $\hat{\Sigma}$ are extended along $C$ by the integral curve of $\frac{\partial}{\partial t}$. For $\Omega=1$, this reduces to equation (3.5) of \cite{Wang-Yau2}. 
\begin{proposition}\label{proposition_mean_curvature_projection}
Along $C$,
\begin{equation}\label{relation_mean_curvature}
\widehat H = -\langle H_0, \breve e_3 \rangle - \frac{\Omega}{ \sqrt{1+\Omega^2| \nabla \tau|^2}} \alpha_{\breve e_3}(\nabla \tau).
\end{equation}
\end{proposition}
\begin{proof} Note that $\breve e_3$ is the unit outward normal of the timelike hypersurface $C$. Denote by $\pi(\cdot, \cdot)=\langle  D_{(\cdot)} \breve e_3 , \cdot \rangle$ the second fundamental form of $C$ with 
respect to $\breve e_3$.
The idea of the proof is to compute the trace of $\pi$ along $C$ in two different tangent frames of $C$, $\{\hat X_*(\frac{\partial}{\partial v^a}), \frac{\frac{\partial}{\partial t}}{\Omega} \}$ and $\{X_*(\frac{\partial}{\partial v^a}),  \breve e_4 \}$. Thus
\[  \hat{\sigma}^{ab}\pi(\hat X_*(\frac{\partial}{\partial v^a}), \hat X_*(\frac{\partial}{\partial v^b})) -  \frac{1}{ \Omega^2}  \pi( \frac{\partial}{\partial t }, \frac{\partial}{\partial t })=
  {\sigma}^{ab}\pi(X_*(\frac{\partial}{\partial v^a}), X_*(\frac{\partial}{\partial v^b})) -  \pi( \breve  e_4 ,\breve  e_4 ).   \\
\]

By definition,  $ \hat{\sigma}^{ab}\pi(\hat X_*(\frac{\partial}{\partial v^a}), \hat X_*(\frac{\partial}{\partial v^b}) )= \widehat H$ and ${\sigma}^{ab}\pi(X_*(\frac{\partial}{\partial v^a}), X_*(\frac{\partial}{\partial v^b}) )=-\langle H_0, \breve e_3 \rangle $. On the other hand, by \eqref{christoffel}

\begin{equation}\label{2nd_C} \pi(\frac{\partial}{\partial t }, \frac{\partial}{\partial t })=-\frac{\breve{e}_3(\Omega)}{\Omega}, \pi(\hat X_*(\frac{\partial}{\partial v^a}), \hat X_*(\frac{\partial}{\partial v^b}) )=\hat h_{ab}, \text{  and  }   \pi(\frac{\partial}{\partial t}, \hat X_*(\frac{\partial}{\partial v^a}))=0.\end{equation}

We use \eqref{decompose_4} and  \eqref{decompose_t} to compute:
\[
\begin{split}
 &  \frac{1}{ \Omega^2}  \pi( \frac{\partial}{\partial t }, \frac{\partial}{\partial t })-  \pi( \breve  e_4 ,\breve  e_4)    \\
= &  \frac{1}{ \Omega^2}  \pi(\frac{\partial}{\partial t },\frac{\partial}{\partial t }) - \frac{\Omega}{\sqrt{1+\Omega^2| \nabla \tau|^2}}\pi({\nabla \tau} ,\breve  e_4)   - \frac{1}{\Omega\sqrt{1+\Omega^2| \nabla \tau|^2}}\pi(\frac{\partial}{\partial t } ,\breve  e_4 ) \\
=& - \frac{\Omega}{ \sqrt{1+\Omega^2| \nabla \tau|^2}} \langle D_{\nabla \tau} \breve e_3  , \breve e_4 \rangle  -  \pi({\frac{\partial}{\partial t }}, \hat \nabla \tau ).
\end{split}\] $\pi({\frac{\partial}{\partial t }}, \hat \nabla \tau)$ vanishes by \eqref{2nd_C}.

\end{proof}
In addition, we derive an identity for the connection one-form $\alpha_{\breve e_3}$ on $\Sigma$ that relates it to the second fundamental form of $\widehat \Sigma$.

\begin{proposition} \label{connection_reference}
Along $C$, the connection one-form $\alpha_{\breve e_3}$ on $\Sigma$ satisfies
\begin{equation}\label{connection_reference_one}
(\alpha_{\breve e_3})_a =  \sqrt{1+\Omega^2| \nabla \tau|^2} ( \Omega \hat \nabla^b \tau  \hat h_{ab} - \breve e_3 (\Omega) \tau_a )\end{equation}
where $\hat h_{ac}$ on the right hand side is the extension of the second fundamental form of $\widehat \Sigma$ to $C$ by the one-parameter family $\phi_t$.
\end{proposition}
\begin{proof}
By definition, $(\alpha_{\breve e_3})_a$ is
\[
\begin{split}
&  \pi( X_*(\frac{\partial}{\partial v^a}) , \breve e_4) \\
=&  \pi( \hat X_*(\frac{\partial}{\partial v^a}) + \tau_a \frac{\partial}{\partial t}  , \frac{ \sqrt{1+\Omega^2| \nabla \tau|^2}}{\Omega}\frac{\partial}{\partial t} + \Omega \sqrt{1+\Omega^2| \nabla \tau|^2} \hat \nabla \tau)\\
=& - \breve e_3 (\Omega) \tau_a  \sqrt{1+\Omega^2| \nabla \tau|^2} + \Omega  \sqrt{1+\Omega^2| \nabla \tau|^2} (\hat \nabla^b \tau)  \hat h_{ab}.
\end{split}
\]
\eqref{decompose_4} is used in  the first equality, and \eqref{christoffel} and \eqref{relation_gradient} are used in the second equality. \end{proof}
We have the following lemma for the restriction of the static potential to surfaces in the static slice. 
\begin{lemma}
Let $\Sigma$ be a surface in the static slice. Let $\Delta$ be the Laplace operator of the induced metric, $\breve e_3$ be the unit outward normal, $H_0$ be the mean curvature, and $h_{ab}$ be the second fundamental form. We have
\begin{align}
\label{induced_laplace}(\Delta + 2 \kappa) \Omega = & - H_0 \breve e_3(\Omega)\\
\label{mix_derivative}\nabla_a \breve e_3(\Omega)  = & h_{ab} \nabla^b \Omega.
\end{align}
\end{lemma}
\begin{proof}
Both equations are simple consequences of  \eqref{static_potential_hessian} and the definition of the second fundamental form.
\end{proof}
\section{A conservation law}
Proposition \ref{proposition_mean_curvature_projection} leads to the following conservation law for surfaces in the dS or AdS spacetime. This generalizes Proposition 3.1 of \cite{Wang-Yau2}.
\begin{proposition}\label{conservation}
For any surface $\Sigma$ in the reference spacetime, we have the following conservation law:
\[
\int \Omega \widehat H d \widehat \Sigma = \int  \left[ - \Omega\sqrt{1+\Omega^2| \nabla \tau|^2}  \langle  H_0, \breve e_3 \rangle - \Omega^2  \langle D_{\nabla \tau} \breve e_3  , \breve e_4 \rangle \right] d \Sigma.
\]
\end{proposition} 
\begin{proof}
Multiply \eqref {relation_mean_curvature} by $\Omega$ and integrate over $\Sigma$. By  \eqref{relation_lower_metric}, the two area forms satisfy
\begin{equation}\label{area_form}  d \widehat \Sigma =  \sqrt{1+\Omega^2| \nabla \tau|^2}  d  \Sigma . \end{equation}
\end{proof}
To define the quasi-local energy, the right hand side of the conservation law is rewritten in terms of the mean curvature gauge in the following proposition. 
\begin{proposition} \label{total_mean_mean_gauge}
In terms of the connection one-form in mean curvature gauge $\alpha_{H_0}$, the conservation law in Proposition \ref{conservation} reads
\[
\begin{split}
& \int \Omega \widehat H d \widehat \Sigma = \int  \Big [ \sqrt{(1+\Omega^2| \nabla \tau|^2) |H_0|^2  \Omega^2 + div(\Omega^2 \nabla \tau)^2 } + div(\Omega^2 \nabla \tau) \theta  -  \alpha_{H_0} (\Omega^2 \nabla \tau)  \Big ] d \Sigma,
\end{split}
\] where \begin{equation}\label{gauge_angle}
\theta = - \sinh^{-1}  \frac{ div(\Omega^2 \nabla \tau) }{|H_0|\Omega \sqrt{1+\Omega^2| \nabla \tau|^2} }.
\end{equation}
\end{proposition}
\begin{proof}
Let $\theta$ be the angle between the oriented frames $\{ -\frac{H}{|H|}, \frac{J}{|H|} \}$ and $ \{ \breve e_3,\breve e_4\}$, i.e.
\begin{equation}
\begin{split} \label{gauge}
- \frac{H_0}{|H_0|} = & \cosh \theta \breve e_3 + \sinh \theta \breve e_4 \\
\frac{J_0}{|H_0|} = & \sinh \theta \breve e_3 + \cosh \theta \breve e_4. \\
\end{split}
\end{equation}

In particular, we have
\begin{equation} \label{gauge_change}   \langle H_0 , \breve e_4 \rangle  = |H_0| \sinh \theta, \quad
  - \langle H_0 , \breve e_3 \rangle  = |H_0| \cosh \theta, 
 \text{ \,\, and  \,\,} 
 \alpha_{H_0} = \alpha_{\breve e_3} + d \theta
.\end{equation} 
To compute $\langle H_0 , \breve e_4 \rangle$, we start with $\langle D_{e_a} \frac{\partial}{\partial t} ,  e_a \rangle =0$ and then use \eqref{decompose_t} to derive
\[ 
\Omega \sqrt{1+\Omega^2| \nabla \tau|^2} \langle D_{e_a}  \breve e_4,  e_a \rangle = \langle D_{e_a}  \Omega^2 \nabla \tau ,  e_a \rangle.
\]
The right hand side is precisely $ div(\Omega^2 \nabla \tau)$. As a result,
\[  - \langle H_0 , \breve e_4 \rangle  =   \frac{ div(\Omega^2 \nabla \tau) }{\Omega \sqrt{1+\Omega^2| \nabla \tau|^2} }\] and $\theta$ is given by \eqref{gauge_angle}.
The proposition now follows from a direct computation.
\end{proof}
\section{Definition of the Quasi-local energy}
Now we consider a surface $\Sigma$ in a general spacetime $N$.  As in \cite{Wang-Yau2,Wang-Yau3}, a quasi-local energy is assigned to each pair of an isometric embedding $X$ of $\Sigma$ into the reference spacetime, and an observer $T_0$ (a future timelike Killing field).  Isometric embeddings into the dS spacetime and the AdS spacetime are studied in \cite{Lin-Wang}. The set of observers is simply the orbit of $\frac{\partial}{\partial t}$ under the isometry group of the reference spacetime. 
See Section 7.2 for more details in the AdS case.

Let $\Sigma$ be  a surface  in a  spacetime $N$. We assume the mean curvature vector $H$ of $\Sigma$ is spacelike and the normal bundle of $\Sigma$ is oriented. The data we use for defining the quasi-local energy is the triple $(\sigma,|H|,\alpha_H)$ where $\sigma$ is the induced metric, $|H|$ is the norm of the mean curvature vector, and $\alpha_H$ is the connection one-form of the normal bundle with respect to the mean curvature vector
\[ \alpha_H(\cdot )=\langle \nabla^N_{(\cdot)}   \frac{J}{|H|}, \frac{H}{|H|}   \rangle.  \]
Here $J$ is the reflection of $H$ through the incoming  light cone in the normal bundle. For an isometric embedding $X$ into the reference spacetime, we write $X= (\tau ,X^1,X^2,X^3)$ with respect to a fixed static chart of the reference spacetime. The quasi-local energy associated to the pair $(X,\frac{\partial}{\partial t})$ is defined to be
\begin{equation}\label{energy_fix_chart_base}
\begin{split}
  E(\Sigma, X,\frac{\partial}{\partial t})
= & \frac{1}{8 \pi}  \Big \{  \int \Omega \widehat H d \widehat \Sigma -
 \int  \Big [ \sqrt{(1+\Omega^2| \nabla \tau|^2) |H|^2  \Omega^2 + div(\Omega^2 \nabla \tau)^2 }  \\
& \qquad -   div(\Omega^2 \nabla \tau)  \sinh^{-1} \frac{ div(\Omega^2 \nabla \tau) }{\Omega |H|\sqrt{1+\Omega^2| \nabla \tau|^2} }
  - \Omega^2 \alpha_{H} (\nabla \tau)  \Big ] d \Sigma
\Big \}.
\end{split}
\end{equation}
Using Proposition \ref{total_mean_mean_gauge}, we have
\begin{equation}\label{energy_fix_chart_graph} 
\begin{split}
   E(\Sigma, X,\frac{\partial}{\partial t}) 
=& \frac{1}{8 \pi}  \Big \{  
 \int  \Big [ \sqrt{(1+\Omega^2| \nabla \tau|^2) |H_0|^2 \Omega^2 + div(\Omega^2 \nabla \tau)^2 }  \\
& \qquad -   div(\Omega^2 \nabla \tau)  \sinh^{-1} \frac{ div(\Omega^2 \nabla \tau) }{\Omega |H_0|\sqrt{1+\Omega^2| \nabla \tau|^2} }
  - \Omega^2 \alpha_{H_0} (\nabla \tau)  \Big ] d \Sigma\\
& -  \int  \Big [ \sqrt{(1+\Omega^2| \nabla \tau|^2) |H|^2  \Omega^2 + div(\Omega^2 \nabla \tau)^2 }  \\
& \qquad -   div(\Omega^2 \nabla \tau)  \sinh^{-1} \frac{ div(\Omega^2 \nabla \tau) }{\Omega |H|\sqrt{1+\Omega^2| \nabla \tau|^2} }
  - \Omega^2 \alpha_{H} (\nabla \tau)  \Big ] d \Sigma
\Big \}.
\end{split}
\end{equation}
\begin{remark} \label{Brown-York_positive}
For an isometric embedding into the static slice of the AdS spacetime, 
\[ E(\Sigma,X,\frac{\partial}{\partial t}) = \int \Omega(H_0-|H|) d\Sigma.\] 
Such an expression was studied in \cite{Wang-Yau07, Shi-Tam}. In particular, the positivity of the above expression was proved in \cite{Shi-Tam}.
\end{remark}
While the above expression seems to depend on the choice of the static chart, we can rewrite it purely in terms of the isometric embedding $X$ and the observer $T_0$. In fact, $\Omega^2 = - \langle T_0,T_0 \rangle$ and  $- \Omega^2  \nabla \tau= T_0^\top $, the tangential component of $T_0$ to $X(\Sigma)$.  Thus
\begin{definition}\label{energy_invariant} 
The quasi-local energy $E(\Sigma, X,T_0)$ of $\Sigma$ with respect to the pair $(X,T_0)$ of an isometric embedding $X$ and an observer $T_0$ is  
\[ 
\begin{split}
   & 8 \pi E(\Sigma, X,T_0) \\
=&  
 \int_{\Sigma}  \Big [ \sqrt{  - \langle T_0^\perp,T_0^\perp \rangle |H_0|^2  + div(T_0^\top)^2 }  -   div(T_0^\top)  \sinh^{-1} \frac{ div(T_0^\top) }{|H_0|\sqrt{-  \langle T_0^\perp,T_0^\perp \rangle} }   + \alpha_{H_0} (T_0^\top)  \Big ]  d\Sigma\\
   & - \int_{\Sigma}  \Big [ \sqrt{ -  \langle T_0^\perp,T_0^\perp \rangle |H|^2  + div(T_0^\top)^2 }  -   div(T_0^\top)  \sinh^{-1} \frac{ div(T_0^\top) }{|H|\sqrt{-  \langle T_0^\perp,T_0^\perp \rangle} }   + \alpha_{H} (T_0^\top)  \Big ]d\Sigma .
\end{split}
\]
where $T_0^\perp$ is the normal part of $T_0$ to $X(\Sigma)$.
\end{definition}

The quasi-local energy is invariant with respect to the isometry of the reference spacetime if an isometry is applied to both $X$ and $T_0$.  As a result, in studying the variation of $E$, it suffices to consider the quasi-local energy with respect to a fixed $T_0 =  \frac{\partial}{\partial t}$. 

The quasi-local energy is expressed in terms of the difference of two integrals. We refer to the first integral in \eqref{energy_fix_chart_base} as the reference Hamiltonian and the second integral in \eqref{energy_fix_chart_base} as the physical Hamiltonian.
\section{First variation of the quasi-local energy}
In this section, we compute the first variation of the quasi-local energy. It suffices to consider the variation of the isometric embedding $X$ while fixing $T_0=\frac{\partial}{\partial t}$.

\begin{definition}
An optimal isometric embedding for the data $(\sigma, |H|, \alpha_H)$ is an isometric embedding $X_0$ of $\sigma$ into the reference spacetime (dS or AdS) that is a critical point of the quasi-local energy $E(\Sigma, X, \frac{\partial}{\partial t})$ among all nearby isometric embeddings $X$ of $\sigma$.
\end{definition}

For the Wang-Yau quasi-local energy with the Minkowski  reference, the first variation of  the quasi-local energy is computed in Section 6 of \cite{Wang-Yau2}. The computation of the variation of the physical Hamiltonian is straightforward  and the main difficulty is to evaluate the variation of the reference Hamiltonian. In \cite{Wang-Yau2}, this is done by computing the variation of the total mean curvature of a surface in $\R^3$ with respect to a variation of the metric. 
This becomes more complicated here since the isometric embedding equation also involves the static potential when the reference is either the dS or AdS spacetime. Instead of following the approach in  \cite{Wang-Yau2}, we derive the first variation by an alternative approach used in \cite{Chen-Wang-Yau1}. The idea there is to consider the image  $X(\Sigma)$ in the reference spacetime as a new physical surface and show that it is naturally a critical point of the quasi-local energy with respect to other isometric embeddings into the reference spacetime. We first derive the following result for surfaces in the reference spacetime. 

\begin{theorem} \label{thm_own_critical}
The identity isometric embedding for a surface $\Sigma$ in the reference spacetime is a critical point of its own quasi-local energy. Namely, suppose $\Sigma$ is in the reference spacetime defined by an embedding $X_0$. Consider a family of isometric embeddings $X(s)$, $-\epsilon<s<\epsilon$ such that $X(0)=X_0$. Then we have
\[  \frac{d}{ds}|_{s=0} E(\Sigma, X(s), \frac{\partial}{\partial t})= 0. \]
\end{theorem}

\begin{proof} Denote $\frac{d}{ds}|_{s=0}$ by $\delta$ and set 
\[\frak{H}_1= \int \Omega \widehat H d \widehat \Sigma \]
and 
\[\begin{split} \frak{H}_2&= \int  \Big [ \sqrt{(1+\Omega^2| \nabla \tau|^2) |H_0|^2  \Omega^2 + div(\Omega^2 \nabla \tau)^2 } \\
&-   div(\Omega^2 \nabla \tau)  \sinh^{-1} \frac{ div(\Omega^2 \nabla \tau) }{\Omega |H_0|\sqrt{1+\Omega^2| \nabla \tau|^2} }   - \Omega^2 \alpha_{H_0} (\nabla \tau)  \Big ] d \Sigma\end{split}.\]
It suffices to prove that $\delta \frak{H}_1=\delta \frak{H}_2$,
where for the variation of $\frak{H}_2$, it is understood that $H_0$ and $\alpha_{H_0}$ are fixed at their values at the initial surface  $X_0(\Sigma)$ and only $\tau$ and $\Omega$ are varied. We compute the variation of $\frak{H}_2$, rewrite it as an integral on the projection $\widehat \Sigma$, and then compare with the variation of $\frak{H}_1$ using the identities in Section 2.  

It is convenient to rewrite   $\frak{H}_1$  and $\frak{H}_2$ in terms of the following two quantities: $A = \Omega \sqrt{1+\Omega^2|\nabla \tau|^2}$ and $B= div(\Omega^2 \nabla \tau)$. In terms of $A$ and $B$ 
\[
\begin{split}
\frak{H}_1= &\int \widehat H A \,d \Sigma\\
 \frak{H}_2=&\int \left[ \sqrt{|H_0|^2A^2+B^2} - B \sinh^{-1} \frac{B}{|H_0| A}  - \alpha_{H_0}(\Omega^2 \nabla \tau) \right] d \Sigma. 
\end{split}
\] 
As a result, we have
\[
\begin{split}
 \delta \frak{H}_2 
=&\int \left [\delta A(\frac{|H_0|^2A}{\sqrt{|H_0|^2A^2+B^2}} + \frac{B^2}{A\sqrt{|H_0|^2A^2+B^2}}) \right]d\Sigma 
-\int \left[ (\delta B)  \sinh^{-1} \frac{B}{|H_0|A}+\alpha_{H_0}(\delta( \Omega^2 \nabla \tau)) \right] d \Sigma\\
=&\mbox{I}-\mbox{II}
\end{split}
\]

By \eqref{gauge_change} and $\sinh \theta = -\frac{B}{|H_0|A}$, integrating by parts gives
\[\mbox{II}=\int [\delta B (-\theta)+\alpha_{H_0}(\delta( \Omega^2 \nabla \tau))] d\Sigma =\int \left[ \delta(\Omega^2 \nabla\tau)\cdot \nabla\theta+\alpha_{H_0}(\delta( \Omega^2 \nabla \tau))\right] d\Sigma=\int \alpha_{\breve e_3} (\delta( \Omega^2 \nabla \tau)) d\Sigma.
\]

On the other hand, we simplify the integrand of $\mbox{I}$ using \eqref{gauge_angle}, \[\frac{|H_0|^2A}{\sqrt{|H_0|^2A^2+B^2}} + \frac{B^2}{A\sqrt{|H_0|^2A^2+B^2}}=\frac{\sqrt{|H_0|^2A^2+B^2}}{A}=-\langle H_0, \breve{e}_3\rangle.\]

Therefore, by  \eqref{relation_mean_curvature}, $\mbox{I}$ is equal to 
\[
\begin{split}
& \int   (-\langle H_0, \breve e_3\rangle)\delta A d\Sigma \\
=& \int   [\widehat H + \frac{\Omega \, \alpha_{\breve e_3} (\nabla \tau)  }{ \sqrt{1+\Omega^2| \nabla \tau|^2}}]\delta A d\Sigma\\
= & \int \widehat H\delta Ad\Sigma + \int   \left[\frac{(\delta \Omega) \Omega^3|\nabla \tau|^2 + \Omega^4 \nabla \tau \nabla \delta \tau}{1+\Omega^2|\nabla \tau|^2}( \alpha_{\breve{e}_3}(\nabla\tau)) +( \delta \Omega)  \Omega  \alpha_{\breve{e}_3}(\nabla\tau) \right] d\Sigma.
\end{split}
\]
and 
\begin{equation} \label{del_H2}
\begin{split}
 \delta\frak{H}_2&=  \int \widehat H\delta Ad\Sigma  + \int \left[  \frac{(\delta \Omega) \Omega^3|\nabla \tau|^2 + \Omega^4 \nabla \tau \nabla \delta \tau}{1+\Omega^2|\nabla \tau|^2}( \alpha_{\breve{e}_3}(\nabla\tau)) 
- \alpha_{\breve{e}_3}(\Omega\delta \Omega  \nabla\tau + \Omega^2 \nabla \delta \tau)\right] d\Sigma\\
&=  \int \widehat H\delta Ad\Sigma  - \int  (\alpha_{\breve{e}_3})_a (\sigma^{ac} -\frac{\Omega^2 \nabla^a \tau \nabla^c \tau}{1+\Omega^2|\nabla \tau|^2})(\Omega \delta \Omega  \tau_c+ \Omega^2 \delta \tau_c) d\Sigma \\
&=  \int \widehat H\delta Ad\Sigma  +\int - (\alpha_{\breve{e}_3})_a \hat \sigma^{ac}( \Omega\delta \Omega  \tau_c+ \Omega^2 \delta \tau_c) d\Sigma.
\end{split}
\end{equation}
Applying  Proposition \ref{connection_reference}, the second integral in the last line can be rewritten as

\[
\begin{split}
& \int   \sqrt{1+\Omega^2| \nabla \tau|^2}( \breve e_3 (\Omega) \tau_a - \Omega \hat  \nabla^b \tau \hat h_{ab}) \hat \sigma^{ac}(\Omega \delta \Omega  \tau_c+ \Omega^2 \delta \tau_c) d \Sigma  \\
= & \int [\breve e_3(\Omega)\hat \sigma^{ab}- \Omega \hat h^{ab} ]  ( \Omega \delta \Omega  \tau_a \tau_b + \Omega^2 \tau_a \delta \tau_b)  d \widehat \Sigma \\
= &  \frac{1}{2}\int [\breve e_3(\Omega)\hat \sigma^{ab}- \Omega \hat h^{ab} ]  (\delta \hat  \sigma)_{ab}  d \widehat \Sigma.
\end{split}
\]

On the other hand,  as $\Omega d\widehat\Sigma=A d\Sigma$ and $\delta d\Sigma=0$,
\begin{equation}\label{del_H1}
\begin{split}
\delta\frak{H}_1=   \int \widehat H\delta Ad\Sigma  +  \int \Omega \delta \widehat H   d \widehat \Sigma . \\
\end{split}
\end{equation}

To prove $\delta \frak{H}_1=\delta\frak{H}_2$, by  \eqref{del_H2} and \eqref{del_H1}, it suffices to show
\begin{equation}\label{eq_hat}\int \Omega \left[ \delta \widehat H +\frac{1}{2}  \hat h^{ab} (\delta \hat  \sigma)_{ab}\right]d \widehat \Sigma=\frac{1}{2}\int \left[\breve e_3(\Omega)\hat \sigma^{ab}  (\delta \hat  \sigma)_{ab}\right]  d \widehat \Sigma.\end{equation}

We decompose $\delta \widehat X$ into tangential and normal parts to $\widehat \Sigma$. Let
\[ \delta \widehat X = \alpha^a \frac{\partial \widehat X}{\partial v^a} + \beta \nu.  \]
For the first and second variations of the induced metric (see \cite[Section 6]{Wang-Yau2} for the Euclidean case), we have
\begin{align}
\label{first_variation_metric} (\delta \hat \sigma)_{ab}   =&2 \beta \hat h_{ab} + \hat \nabla_a( \alpha^c \hat \sigma_{cb}) + \hat \nabla_b (\alpha^c \hat \sigma_{ca}) \\
\label{second_variation_metric_1}  \delta \widehat H =& - \hat h^{ab}  (\delta \hat \sigma)_{ab} - \widehat  \Delta \beta  - 2 \kappa \beta +  \hat h_{ac} \hat \nabla^a \alpha^c +\beta \hat \sigma^{ab} \hat \sigma^{dc} \hat h_{ac} \hat h_{bd}+\hat \nabla^b (\alpha^c \hat h_{bc}).
\end{align}

We derive from \eqref{first_variation_metric} and \eqref{second_variation_metric_1}
\begin{equation}
\label{second_variation_metric}  \delta \widehat H +\frac{1}{2} \hat h^{ab}  (\delta \hat \sigma)_{ab} =- \widehat \Delta \beta - 2 \kappa \beta +   \hat \nabla^b( \alpha^c \hat h_{cb}).\end{equation}
\eqref{eq_hat} is thus equivalent to 
\[
\begin{split}
\int  \Omega  [- \widehat \Delta \beta - 2 \kappa \beta +  \hat \nabla^b( \alpha^c \hat h_{cb})] d \widehat \Sigma=\int \breve e_3(\Omega) [ \beta \widehat H + \hat \nabla^b( \alpha^c \hat \sigma_{cb})]  d \widehat \Sigma . 
\end{split}
\]
The above equality follows from the following two identities:
\begin{align}
\label{equality_b}\int \breve e_3(\Omega)  \beta \widehat H  d \widehat \Sigma = & \int  \Omega  [- \widehat \Delta \beta - 2 \kappa \beta]  d \widehat \Sigma \\
\label{equality_a}\int \breve e_3(\Omega) \hat \nabla^b( \alpha^c \hat \sigma_{cb}) d \widehat \Sigma = & \int  \Omega    \hat \nabla^b ( \alpha^c \hat h_{cb}) d \widehat \Sigma,
   \end{align} which can be derived by integrating by parts and applying \eqref{induced_laplace} and \eqref{mix_derivative}. \end{proof}
   
\begin{definition} The quasi-local energy density with respect to $(X, T_0)$ is defined to be 
\begin{equation} \label{rho} \begin{split}\rho &= \frac{\sqrt{|H_0|^2 +\frac{(div \Omega^2 \nabla \tau)^2}{\Omega^2+\Omega^4 |\nabla \tau|^2}} - \sqrt{|H|^2 +\frac{(div \Omega^2 \nabla \tau)^2}{\Omega^2+\Omega^4 |\nabla \tau|^2}} }{ \Omega\sqrt{1+ \Omega^2|\nabla \tau|^2}}. \end{split}\end{equation}
\end{definition}
We derive the following formula for the first variation of the quasi-local energy.
\begin{theorem} \label{thm_first_variation_graph}
Let $\Sigma$ be a surface in a physical spacetime with the data  $(\sigma,|H|, \alpha_H)$. Let $X_0$ be an isometric embedding of $\sigma$ into the reference spacetime and let $(|H_0|, \alpha_{H_0})$
 be the corresponding data on $X_0(\Sigma)$. Consider a family of isometric embeddings $X(s)$, $-\epsilon<s<\epsilon$ such that $X(0)=X_0$. Then we have
\begin{equation}\label{first_variation_graph}
\begin{split}
   & \frac{d}{ds}|_{s=0} E(\Sigma, X(s),\frac{\partial}{\partial t}) \\
=& \frac{1}{8 \pi} \int_{\Sigma} (\delta \tau)   div\left [  \Omega^2 \nabla \sinh^{-1} \frac{\rho div (\Omega^2 \nabla \tau)}{|H_0||H|}  - \rho \Omega^4 \nabla \tau +\Omega^2(\alpha_{H_0} - \alpha_H) \right ]   d \Sigma \\
 & +\frac{1}{8 \pi} \int_{\Sigma} \delta X^i \bar\nabla_i \Omega \left [ \rho  \Omega (1+ 2 \Omega^2|\nabla \tau|^2)  -2  \Omega \nabla \tau \nabla \sinh^{-1} \frac{\rho div (\Omega^2 \nabla \tau)}{|H_0||H|}  + (\alpha_{H} - \alpha_{H_0})(2 \Omega \nabla \tau)  \right ]   d \Sigma,
\end{split}  
\end{equation} where $\delta\tau=\frac{d}{ds}|_{s=0} \tau(s)$ and $\delta X^i=\frac{d}{ds}|_{s=0} X^i(s)$.
\end{theorem}
\begin{proof}
 Let $A = \Omega \sqrt{1+\Omega^2|\nabla \tau|^2}$ and $B= div(\Omega^2 \nabla \tau)$.
In terms of $A$ and $B$, 
\[  \rho = \frac{\sqrt{A^2 |H_0|^2+B^2}-\sqrt{A^2|H|^2+B^2}}{A^2}.\]

Write 
\[\begin{split}
 & 8 \pi  E(\Sigma, X(s),\frac{\partial}{\partial t})- 8\pi E( X_0(\Sigma), X(s),\frac{\partial}{\partial t})\\
=& \int  \Big [ \sqrt{A^2(s) |H_0|^2+B^2(s)} -   B(s)  \sinh^{-1} \frac{B(s)}{ |H_0| A(s) }
  - \Omega^2(s) \alpha_{H_0} (\nabla \tau(s))  \Big ] d \Sigma\\
-&\int  \Big [ \sqrt{A^2(s)  |H|^2+B^2(s) } -  B(s)  \sinh^{-1} \frac{B(s) }{ |H| A(s) }
  - \Omega^2(s) \alpha_{H} (\nabla \tau(s))  \Big ] d \Sigma,\\
\end{split}\] where $A(s) = \Omega(s) \sqrt{1+\Omega^2(s)|\nabla \tau(s)|^2}$ and $B(s)= div(\Omega^2(s) \nabla \tau(s))$. By Theorem \ref{thm_own_critical},  $\delta  E( X_0(\Sigma), X,\frac{\partial}{\partial t})=0$. Therefore,  in terms of $A$ and $B$, $8 \pi  \delta E(\Sigma, X,\frac{\partial}{\partial t})$ is equal to
\[
\begin{split}
&\int (\delta A)(\frac{\sqrt{A^2 |H_0|^2  + B^2 }-\sqrt{A^2|H|^2  \Omega^2 + B)^2}}{A}) d \Sigma\\
+& \int (\delta B) ( \sinh^{-1} \frac{B}{|H|A} -\sinh^{-1} \frac{B}{|H_0|A} ) + (\alpha_H-\alpha_{H_0})(2 \Omega \delta \Omega \nabla \tau + \Omega^2 \nabla \delta \tau) d \Sigma  
\end{split}
\]
A direct computation shows that 
\[  \sinh^{-1} \frac{B}{|H|A} -\sinh^{-1} \frac{B}{|H_0|A}  = \sinh^{-1}[\frac{B}{|H||H_0|A^2} (\sqrt{A^2 |H_0|^2+B^2}-\sqrt{A^2|H|^2+B^2}) ].\]

On the other hand, 
\[  
\begin{split}
\delta A = &(\delta \Omega) \frac{1+ 2\Omega^2 |\nabla \tau|^2}{\sqrt{1+ \Omega^2 |\nabla \tau|^2}} + \frac{\Omega^3 \nabla \tau \nabla \delta \tau}{\sqrt{1+ \Omega^2 |\nabla \tau|^2}}\\
\delta B= & div(2 \Omega \delta \Omega \nabla \tau + \Omega^2 \nabla \delta \tau).
\end{split}
\]
The theorem follows from integration by parts, collecting terms, and $ \delta \Omega = \delta X^i \bar \nabla_i \Omega$.
\end{proof}
$\delta X^i$ and $\delta \tau$ are constrained by the linearized isometric embedding equation 
\[  \delta X^i \bar \nabla_i \Omega^2 \tau_a \tau_b+ \Omega^2 (\tau_a\delta \tau_b + \tau_b \delta \tau_a) = g_{ij}X^i_a \delta X^j_b + \delta X^k \partial_k g_{ij}X^i_a  X^j_b .\]

\section{Second variation and local minimum of the quasi-local energy}
First, we prove the following lemma about surfaces in the static slice of the reference spacetime. A similar and related inequality was obtained in \cite{Kwong-Miao}. 
\begin{lemma} \label{Reilly-Positivity}
Let $\Sigma$ be a  convex surface  in the static slice of the reference spacetime. Let $H_0$ and $h_{ab}$  be the mean curvature and second fundamental form of $\Sigma$. Then for any smooth function $f$ on $\Sigma$, the integral 
\begin{equation}  \label{second_variation_positive}\int \left \{ \frac{[div(\Omega^2 \nabla f)]^2}{H_0 \Omega} - \Omega^3 h^{ab} f_af_b + \Omega^2 |\nabla f|^2 e_3 (\Omega) \right \} d\Sigma\end{equation} is non-negative and vanishes if and only if $f$ can be smoothly extended to a smooth function $\bar{f}$ in the region enclosed by $\Sigma$
that satisfies
\begin{equation}\label{static} \bar \nabla^2 (\bar{f}\Omega) + \kappa (\bar{f}\Omega) g=0.\end{equation}
In particular, $\bar{f}\Omega$ is another static potential \eqref{static_potential_hessian}.
\end{lemma}
\begin{proof}
Let $f =\frac{F}{\Omega}$ and $\nabla f = \frac{\nabla F}{\Omega} - \frac{F \nabla \Omega}{ \Omega ^2}$. We compute
\begin{equation}\label{1st_term}\int [\frac{[div(\Omega^2 \nabla f)]^2}{H_0 \Omega}]d\Sigma=\int [\frac{1}{H_0\Omega} (\Omega\Delta F-F\Delta\Omega)^2]d\Sigma.\end{equation}

On the other hand, 
\[  
- \Omega^3 h^{ab} f_af_b =- \Omega h^{ab} F_aF_b +h^{ab}\Omega_b(- \frac{F^2}{\Omega} \Omega_a + 2 FF_a).  
\]
Using \eqref{mix_derivative}, $h^{ab}\Omega_b=\nabla_b e_3(\Omega)$, and integrating by parts on $\Sigma$, we obtain 
\begin{equation}\label{2nd_term}  \int ( - \Omega^3 h^{ab} f_af_b ) d\Sigma  
= \int \left\{- \Omega h^{ab} F_aF_b +e_3(\Omega) div (\frac{F^2\nabla \Omega}{\Omega}-2 F\nabla F)   \right\}d\Sigma. 
\end{equation}

Let $M$ be the region on the static slice enclosed by $\Sigma$. We need the following  Reilly formula with static potential $\Omega$ from \cite{Qiu-Xia} (see also \cite{Li-Xia}):
\[
\begin{split}
   & \int_M \Omega [(\bar \Delta \bar{F} + 3 \kappa \bar{F})^2  - | \bar \nabla ^2 \bar{F} + \kappa \bar{F} g|^2] \\
=& \int \left\{ \Omega (2 e_3(F) \Delta F + H_0 e_3(F)^2 + h^{ab} F_aF_b + 4\kappa F e_3(F)) + e_3(\Omega) (|\nabla F|^2 - 2 \kappa F^2)\right\}d\Sigma,
\end{split}
\] where $\bar{F}$ is a smooth extension of $F$ to $M$. 
Extending $F$ by solving the elliptic PDE $\bar \Delta \bar{F} + 3 \kappa \bar{F} = 0$ with boundary data $F$ on $\Sigma=\partial M$ (see \cite{Kwong-Miao} for the solution of the Dirichlet boundary value problem), we have
\[ - \int (\Omega h^{ab} F_aF_b)d \Sigma \ge \int \left\{\Omega [2 e_3(F) \Delta F + H_0 e_3(F)^2  + 4\kappa  F e_3(F)] + e_3(\Omega) (|\nabla F|^2 - 2 \kappa F^2)\right\}d\Sigma.\]

Plugging this into \eqref{2nd_term} and expanding $\Omega^2 |\nabla f|^2 e_3 (\Omega)$ by replacing  $f =\frac{F}{\Omega}$, we obtain

\begin{equation}\label{2nd+3rd}\begin{split} &\int \left \{  - \Omega^3 h^{ab} f_af_b + \Omega^2 |\nabla f|^2 e_3 (\Omega) \right \} d\Sigma\\
&\geq  \int \left\{\Omega [H_0 e_3(F)^2  + 2 e_3(F)(\Delta F+2\kappa F)] + e_3(\Omega) [\frac{F^2}{\Omega}( \Delta\Omega- 2 \kappa \Omega)-2F\Delta F]\right\} d\Sigma.
\end{split}\end{equation}

Replacing $e_3(\Omega) $ by $-\frac{\Delta \Omega + 2 \kappa \Omega}{H_0}$ by \eqref{mix_derivative}, we arrive at 
\[e_3(\Omega) [\frac{F^2 }{\Omega} (\Delta\Omega- 2 \kappa \Omega)-2F\Delta F]=\frac{1}{H_0\Omega}[-F^2(\Delta\Omega)^2+4\kappa^2\Omega^2F^2
+2F\Omega \Delta F\Delta \Omega+4\kappa \Omega^2 F\Delta F].
\] Plugging this into \eqref{2nd+3rd} and recalling \eqref{1st_term}, the integral in question, after completing squares, is equal to 
\[\int \frac{\Omega}{H_0}[\Delta F+2\kappa F+H_0 e_3(F)]^2 d\Sigma\] which is non-negative.  It is clear that when the equality holds,
\[ \bar \nabla^2 (\bar{F}) + \kappa (\bar{F}) g=0,\] and $\bar{f}=\frac{\bar{F}}{\Omega}$ is a smooth extension of $f$. 
\end{proof}
We prove that a convex surface in the static slice of the reference spacetime is a local minimum of its own quasi-local energy. 
\begin{theorem} \label{minimize_self}
Suppose $X(s)=(\tau(s),X^i(s)), s\in (-\epsilon, \epsilon)$ is a family of isometric embeddings of the same metric $\sigma$ into the reference spacetime such that the image of $X(0)$ is a convex surface $\Sigma_0$ in the static slice, then
\[
\begin{split}
\frac{d^2}{ds^2}|_{s=0} E(\Sigma_0,X(s), \frac{\partial}{\partial t}) \ge & 0.\\
\end{split}
\]
In addition, the equality  holds if and only if $f=  \frac{d}{ds}|_{s=0} \tau(s) $ can be smoothly extended to a smooth function $\bar{f}$ in the region enclosed by $\Sigma_0$ that satisfies
\[\bar \nabla ^2 (\bar{f}\Omega) + \kappa (\bar{f}\Omega) g=0.\]
\end{theorem}
\begin{proof}
Let $H_0(X(s))$ and $\alpha_{H_0}(X(s))$ be the mean curvature vector and the connection 1-form in mean curvature gauge of the image of $X(s)$. For simplicity, set $\delta |H_0| = \frac{d}{ds}|_{s=0} |H_0(X(s))|$ and
$\delta \alpha_{H_0} =\frac{d}{ds}|_{s=0} \alpha_{H_0}(X(s))$. Let $\widehat X(s)=(0,X^i(s))$ be the projection of $X(s)(\Sigma)$ onto the static slice. $\widehat X(s)$ is an isometric embedding of the metric 
\[  \hat \sigma(s)_{ab}= \sigma_{ab} + \Omega^2(s) \tau_a(s) \tau_b(s) \]
into the static slice and $ \delta \hat \sigma = \frac{d}{ds}|_{s=0} \hat \sigma(s)=0$, as $\tau(0)=0$.

From the infinitesimal rigidity of the isometric embeddings into space forms \cite{pogorelov} ,  there is a family of isometries $\hat A(s)$ of the static slice with $\hat A(0)=Id$  such that 
\[
   \delta \hat A =  \delta  \widehat X
\]
along the surface $\Sigma_0$. Here we set $ \delta \hat A = \frac{d}{ds}|_{s=0} \hat A (s)$  and $\widehat X = \frac{d}{ds}|_{s=0} \widehat X (s) $. Moreover, there is a family $A(s)$ of isometries of the reference spacetime whose restriction to the static slice is  the family $\hat A (s)$. Consider the following family of isometric embeddings of $\sigma$ into the reference spacetime:
\[  \breve X(s) = A^{-1}(s) X(s). \]
 Suppose $\breve X (s) = (\breve \tau(s) , \breve X^i(s))$ in the fixed static coordinate, we have \begin{equation} \label{match_first} \frac{d}{ds}|_{s=0} \breve X^i (s)= 0.\end{equation}

We claim that 
\begin{equation}\label{equality_second_variation}
\frac{d^2}{ds^2}|_{s=0} E(\Sigma_0,X(s), \frac{\partial}{\partial t}) = \frac{d^2}{ds^2}|_{s=0} E(\Sigma_0,\breve X(s), \frac{\partial}{\partial t}).
\end{equation}

Let $H_0(\breve X(s))$ and $\alpha_{H_0}(\breve X(s))$ be the the mean curvature vector and the connection 1-form in mean curvature gauge of the images of  $\breve X(s)$. 
\begin{equation}\label{global_isometry_invariant}
\begin{split}
|H_0(X(s))| = &  |H_0(\breve X(s))| \\
\alpha_{H_0} (X(s))=&\alpha_{H_0} (\breve X(s))
\end{split}
\end{equation}
since both are invariant under isometries of the reference spacetime.  By \eqref{match_first}, is easy to see that 
\begin{equation} \label{variation_mean_vanish}  \frac{d}{ds}|_{s=0} | \breve H_0(s)| = 0.   \end{equation}

Moreover, while $\breve \tau(s)$ is different to $\tau(s)$, we have 
\begin{equation}\label{time_function_same} \frac{d}{ds}|_{s=0} \breve \tau(s) = \frac{d}{ds}|_{s=0} \tau(s)  =f \end{equation}
sicne $\tau(0) =0$, $A(0)= Id$ and the static slice is invariant under the action of $A(s)$.  

We apply Theorem  \ref{thm_first_variation_graph}  to each of $X(s)(\Sigma)$ and $\breve X(s)(\Sigma)$ and use \eqref{global_isometry_invariant}, \eqref{variation_mean_vanish} and \eqref{time_function_same} to differentiate
\eqref{first_variation_graph} one more time.  Only the derivative of the term $ \frac{1}{8 \pi} \int_{\Sigma} (\delta \tau)   div ( \Omega^2 \alpha_{H_0} )  d \Sigma $ survives after the evaluation at $s=0$.  We thus  conclude that both sides of \eqref{equality_second_variation} are the same as
\[ 
   - \frac{1}{8 \pi} \int  ( \delta \alpha_{H_0} )(\Omega ^2 \nabla f) d\Sigma . 
\]

It suffices to evaluate the second variation with respect to the family $ \breve X(s)$.  Equivalently, we may assume, for simplicity, that $\delta \widehat X =0$. We follow the computation of  $ \delta \alpha_{H_0}$ from \cite{Chen-Wang-Wang}.

 Let $e_3 =-\frac{H_0}{|  H_0|} $ and $e_4 = \frac{ J_0}{|\breve H_0|}$. From  \eqref{gauge_angle}, \eqref{gauge} and  $\delta \widehat X =0$, we derive
\begin{align}
 \frac{d}{ds}|_{s=0}  H_0&= \frac{div(\Omega^2  \nabla f)}{\Omega^2} \frac{\partial}{\partial t}\\
\label{variation of e_3}  \frac{d}{ds}|_{s=0}  e_3 &= - \frac{div(\Omega^2  \nabla f)}{\Omega^2 |H_0|} \frac{\partial}{\partial t}\\
\label{variation of w_a} \frac{d}{ds}|_{s=0}  X_*(\frac{\partial}{\partial v^a})& = f_a  \frac{\partial}{\partial t}.
\end{align}
As a result, we have 
\begin{align}\label{variation of e_4}
\frac{d}{ds}|_{s=0}  e_{4} = \Omega \nabla f - \frac{div(\Omega^2 \nabla f)}{|H_0| \Omega} e_3
\end{align}
which follows from solving the linear system
\[
\begin{split}
 \frac{d}{ds}|_{s=0} \langle e_{4}, X_*(\frac{\partial}{\partial v^a}) \rangle =&0\\
  \frac{d}{ds}|_{s=0} \langle e_{4},e_3 \rangle=&0
\end{split}
\]
 along with  \eqref{variation of e_3} and  \eqref{variation of w_a}.

We are ready to compute the variation of $\alpha_{H_0}$, which is denoted by $\alpha$ in the remaining part of the proof.
\begin{align*}
(\delta \alpha)_a &= \delta \langle D_a e_3,e_{4} \rangle \\
&= \langle D_{\frac{\partial (\delta X)}{\partial v^a} } e_3,e_{4} \rangle + \langle D_a (\delta e_3),e_{4} \rangle + \langle D_a e_3, \delta e_{4} \rangle.
\end{align*}
By (\ref{variation of e_3}) and  (\ref{variation of e_4}), we get
\begin{align*}
(\delta \alpha)_a &= \langle D_{f_a \frac{\partial}{\partial t}}e_3, e_{4}\rangle + \langle D_a \left( - \frac{div(\Omega^2 \nabla f)}{|H_0 |\Omega} \right) e_4, e_4 \rangle + \langle D_a e_3, \Omega \nabla f - \frac{div(\Omega^2 \nabla f)}{|H_0| \Omega} e_3 \rangle\\
&= \nabla_a \left( \frac{ div ( \Omega^2 \nabla f)}{\Omega |H_0|}\right) + \Omega h_{ab} \nabla^b f  - f_a e_3(\Omega).
\end{align*}
As a result,
\[
\begin{split}
    & - \int  ( \delta \alpha_{H_0} )(\Omega ^2 \nabla f) d\Sigma \\
= & - \int   \Omega^2 f^a [\nabla_a \left( \frac{ div ( \Omega^2 \nabla f)}{\Omega |H_0|}\right) + \Omega h_{ab} \nabla^b f  - f_a e_3(\Omega) ]  d\Sigma\\
= & \int \left \{ \frac{[div(\Omega^2 \nabla f)]^2}{|H_0| \Omega} - \Omega^3 h^{ab} f_af_b + \Omega^2 |\nabla f|^2 e_3 (\Omega) \right \} d\Sigma.
\end{split}
\]
The theorem follows from Lemma \ref{Reilly-Positivity}.
\end{proof}

Finally, we evaluate the second variation of the quasi-local energy and show that for surfaces with spherically symmetric data, namely
\[
\sigma = r^2 \tilde \sigma, \quad |H|=c>0 \textrm{\,\, and \,\, } \alpha_H=0,
\]
there is an isometric embedding into the hyperbolic space $\mathbb H^3$ which minimizes the quasi-local energy with AdS spacetime reference.

\begin{theorem} \label{local_min_AdS}Let $\Sigma$ be a surface in spacetime $N$ with data $(\sigma, |H|, \alpha_H)$. We assume  the mean curvature vector $H$ of $\Sigma$ is spacelike and $\alpha_H =0$. Furthermore, we assume that the image of an isometric embedding of $\sigma$ into $\mathbb H^3$ is convex.
\begin{enumerate}
\item[(i)] Suppose the mean curvature of the isometric embedding into $\mathbb H^3$ satisfies $H_0 \ge |H|$. Then there is an isometric embedding into $\mathbb H^3$ which is a critical point of the quasi-local energy with $AdS$ spacetime reference.
\item[(ii)] Suppose the data on $\Sigma$ is spherically symmetric. Then the above critical point is a  local minimum of the quasi-local energy  with $AdS$ spacetime reference.
\end{enumerate}
\end{theorem}
\begin{proof}
For an isometric embedding into $\mathbb H^3$ as a static slice of AdS, $\tau=0$ and the quasi-local energy \eqref{energy_fix_chart_base} is simply
\begin{equation} \label{choice_static_potential}\frac{1}{8 \pi}\int \Omega(H_0 - |H|) d\Sigma. \end{equation}
The isometric embedding into $\mathbb H^3$ is unique up to an isometry of $\mathbb H^3$. If $H_0 =|H|$ everywhere, then the integral vanishes for any choice of static potential. Otherwise, the integral depends on the isometry which may pick up a different choice of the static potentials $\Omega$. The choice of the static potential corresponds to choosing a base point on $\mathbb H^3$. Hence, the integral $\int \Omega(H_0 - |H|) d\Sigma$ becomes a function on $\mathbb H^3$. Assuming $H_0 \ge |H|$ and they are not equal everywhere, this function is positive, proper and convex since the static potentials approach infinity at the infinity of  $\mathbb H^3$. Hence, there is a unique choice of $\Omega$ that minimizes the quasi-local energy among all the static potentials. Equivalently, if we pick a fixed  static potential, then there is a unique isometric embedding $X_0$ such that for any other isometric embedding $X$ into $\mathbb H^3$, we have
\[ E(\Sigma,X, \frac{\partial}{\partial t}) \ge E(\Sigma,X_0, \frac{\partial}{\partial t}).  \]
In particular, for any family, $\hat A(s)$, of isometries of $\mathbb H^3$ with $\hat A(0)=Id$, we have
\begin{equation} \label{minimum_among_iso}
\begin{split}
\int \delta \hat  A (\Omega ) (H_0 - |H|)  d\Sigma = & 0\\
\int \delta^2 \hat  A ( \Omega )(H_0 - |H|)  d\Sigma \ge & 0\\
\end{split}
\end{equation}
where $ \delta \hat  A = \frac{d}{ds}|_{s=0}\hat  A (s)$ and $ \delta^2 \hat  A= \frac{d^2}{ds^2}|_{s=0} \hat A(s) $ are vector fields on $\mathbb H^3$.

Consider a family of isometric embeddings $X(s)=(\tau(s),X^i(s))$ of $\Sigma$ into $AdS$ where $X(0)$ is the above isometric embedding into $\mathbb H^3$. Let $H_0(s)$ and $ \alpha_{H_0}(s)$ be the mean curvature vector and connection 1-form in mean curvature gauge of the image of $X(s)$. For simplicity, let  $\delta \tau =  \frac{d}{ds}|_{s=0} X^0(s)$, $\delta |H_0| = \frac{d}{ds}|_{s=0} |H_0(s)|$ and 
$\delta \alpha_{H_0} =\frac{d}{ds}|_{s=0} \alpha_{H_0(s)}$. Similarly,  set
$\delta \Omega = \frac{d}{ds}|_{s=0} \Omega(s)$ and 
$\delta^2 \Omega = \frac{d^2}{ds^2}|_{s=0} \Omega(s)$.

Let $\widehat X(s)=(0,X^i(s))$ be the projection of $X(s)(\Sigma)$ onto the static slice. $\widehat X(s)$ is an isometric embedding of the metric 
\[  \hat \sigma(s)_{ab}= \sigma_{ab} + \Omega^2(s) \tau_a(s) \tau_b(s) \]
into $\mathbb H^3$. Set $ \delta \hat \sigma = \frac{d}{ds}|_{s=0} \hat \sigma(s)  $ and $ \delta^2 \hat \sigma = \frac{d^2}{ds^2}|_{s=0} \hat \sigma(s).  $
We have
\begin{equation} \label{variation_projection_metric}
\begin{split}
\delta \hat \sigma =& 0\\
\delta^2 \hat \sigma_{ab} =& \Omega^2(0)\delta \tau_a \delta \tau_b.
\end{split}
\end{equation}
The first variation of the quasi-local energy is
\[ \frac{d}{ds}|_{s=0} E(\Sigma_0,X(s), \frac{\partial}{\partial t}) =\frac{1}{8 \pi} \int ( \delta \Omega) (H_0 - |H|) +   \Omega (\delta |H_0|)   d\Sigma.  \]
From \eqref{variation_projection_metric} and the infinitesimal rigidity of isometric embeddings into  $\mathbb H^3$, there is a family $\hat  A(s)$ of isometries of $\mathbb H^3$ with $\hat  A(0)=Id$  such that 
\begin{equation}\label{first_order_matching}
   \delta \hat A =  \delta  \widehat X.\end{equation}
We conclude that the first term vanishes from \eqref{minimum_among_iso}.  We conclude that $\delta |H_0|=0$ as in the proof of Theorem \ref{minimize_self}. This proves part (i).

For part (ii), it is easy to see that for surfaces with spherical symmetric data, the mean curvature of the image of isometric embeddings into $\mathbb H^3$ is constant and the critical point obtained in part (i) is precisely such that $\Omega$ is constant on the image.

As in the proof of Theorem \ref{minimize_self}, we consider the family $A(s)$ of isometries of the AdS spacetime whose restriction to the static slice is  the family $\hat A (s)$ and consider the following family of isometric embeddings of $\sigma$ into the AdS spacetime:
\[  \breve X(s) = A^{-1}(s) X(s). \]

As in the proof of Theorem \ref{minimize_self}, we apply Theorem  \ref{thm_first_variation_graph}  to each of $X(s)(\Sigma)$ and $\breve X(s)(\Sigma)$ and use \eqref{global_isometry_invariant}, \eqref{variation_mean_vanish} and \eqref{time_function_same} to differentiate \eqref{first_variation_graph} one more time.  In this case, $H_0\not= |H|$ at $s=0$, and we derive that
\begin{equation}
   \frac{d^2}{ds^2}|_{s=0} E(\Sigma_0,  X(s), \frac{\partial}{\partial t})
=  \frac{d^2}{ds^2}|_{s=0} E(\Sigma_0, \breve X(s), \frac{\partial}{\partial t}) + \frac{1}{8 \pi} \int  \delta^2 \hat A ( \Omega ) (H_0 - |H|) d\Sigma 
\end{equation}
and
\begin{equation} \label{second_variation_breve}
\begin{split}
    & \frac{d^2}{ds^2}|_{s=0} E(\Sigma_0, \breve X(s), \frac{\partial}{\partial t})\\
=&\frac{1}{8 \pi} \int ( \delta^2 \breve X^i \bar \nabla_i \Omega) (H_0 - |H|) d\Sigma+\frac{1}{8 \pi} \int  \Omega^3 (H_0 -|H|) |\nabla \delta \tau|^2 d\Sigma \\
&+ \frac{1}{8 \pi}\int  (\frac{H_0 -|H|}{\Omega H_0 |H|}) [div(\Omega^2 \nabla \delta \tau)]^2 d\Sigma -\frac{1}{8 \pi} \int  ( \delta \alpha_{H_0} )(\Omega ^2 \nabla \delta \tau) d\Sigma . 
\end{split}
\end{equation}
From \eqref{minimum_among_iso}, we conclude
\[  \int \delta^2 \hat A(\Omega)(H_0-|H|) d \Sigma \ge 0.\]

The second and third term on the right hand side of \eqref{second_variation_breve}  are manifestly non-negative since $H_0 \ge |H|$. The last term is non-negative as in the proof of Theorem \ref{minimize_self}.
It suffices to show that the first term is also non-negative. We decompose $\delta^2\breve X^i$ into its tangential and normal parts to $X(0)(\Sigma)$. Let
\[ \delta^2 \breve X^i = \alpha^a \frac{\partial \widehat X^i(0)}{\partial v^a} + \beta \nu^i.  \] Since $\Omega$ is a constant on the image of $\widehat{X}$, integrating over $\Sigma$ gives
\[
 \int ( \delta^2 \breve X^i \bar \nabla_i \Omega) (H_0 - |H|) d\Sigma = \int  \beta \nu(\Omega) (H_0-|H|) d \Sigma.
\]

In terms of $\alpha$ and $\beta$, the second variation of the isometric embedding equation is
\begin{equation}
\label{gauge_first_variation_metric}2 \beta  h_{ab} +   \nabla_a \alpha_b +   \nabla_b \alpha_a =2 \Omega^2 \delta \tau_a\delta \tau_b.
\end{equation}

Taking the trace  of \eqref{gauge_first_variation_metric} and integrating, we conclude that 
\[  \int \beta H_0 d\Sigma \ge 0. \] In particular, $\int \beta d\Sigma\ge 0$ since $H_0$ is a constant. It  follows that 
\[  \int  \beta \nu(\Omega) (H_0-|H|) d \Sigma \ge 0 \]
since  $\nu(\Omega) $, $H_0$ and $|H|$ are all positive constants.
\end{proof}

\section{Quasi-local/total  conserved quantities}
The reference spacetime admits $10$ dimensional Killing fields. In addition to the quasi-local energy corresponding to observers, a quasi-local conserved quantity corresponding to each Killing field is defined. We follow the approach in \cite{Chen-Wang-Yau3} to use an isometric embedding to transplant Killing fields of the reference spacetime back to the 2-surface of interest in a physical spacetime. The quasi-local energy can be written in terms of the quasi-local energy density $\rho$ in \eqref{rho} and  the quasi-local momentum density $j$, see \eqref{j_momentum} below.
In the second subsection,  we evaluate the limits of the quasi-local conserved quantities on an  asymptotically AdS initial data set and prove that the limits agree with the total conserved quantities of such an initial data. In the third subsection, we show that the limit of the quasi-local energy is the linear function dual to the total conserved quantities and in the last subsection, we compute the evolution of the total conserved quantities under the Einstein equation.

\subsection{Quasi-local conserved quantities}

We rewrite the quasi-local energy in terms of $\rho$ \eqref{rho} using the expression \eqref{energy_fix_chart_graph}. This is a straightforward computation involving only basic identities of the inverse hyperbolic functions. For the case $\Omega=1$, this is carried out in details in Section 4 of \cite{Chen-Wang-Yau1} for the Wang-Yau quasi-local energy.  After some simplifications, the quasi-local energy in terms of $\rho$ is
\[
\begin{split}
  &E(\Sigma, X, T_0)\\
= & \frac{1}{8\pi}\int_\Sigma \left[\rho(\Omega^2+\Omega^4|\nabla\tau|^2)+ div(\Omega^2 \nabla \tau) \sinh^{-1}(\frac{\rho div(\Omega^2 \nabla \tau) }{|H_0| |H|})-\alpha_{H_0}(\Omega^2 \nabla \tau)+\alpha_H(\Omega^2 \nabla \tau)\right] d\Sigma \end{split}
\]
Let $j$ be the quasi-local momentum density one-form:
\begin{equation}\label{j_momentum}  j = \rho \Omega^2 d \tau - d[ \sinh^{-1} (\frac{\rho div  (\Omega^2 \nabla \tau)}{|H_0||H|})]-\alpha_{H_0}  + \alpha_{H}. \end{equation}

We are ready to define the quasi-local conserved quantity with respect to a pair $(X, T_0)$ and a Killing field $K$.
\begin{definition} \label{ql_conserved}The quasi-local conserved quantity of $\Sigma$ with respect to a pair $(X, T_0)$ and a Killing field $K$ in the reference spacetime is 
\begin{equation}\label{qlcq2}E(\Sigma, X, T_0, K)=-\frac{1}{8\pi} \int_\Sigma
\left[ \langle K, T_0\rangle \rho+j(K^\top) \right]d\Sigma\end{equation}  
where $K^\top$ is the tangential part of $K$ to $X(\Sigma)$, and $\rho$ defined in \eqref{rho} and $j$ defined in \eqref{j_momentum}.
\end{definition}
In particular, when $K=T_0$,  $E(\Sigma, X, T_0,T_0)$  recovers the quasi-local energy $E(\Sigma,X,T_0)$ since the tangential part of $T_0$ to $X(\Sigma)$ is $-\Omega^2 \nabla \tau$ and $\langle T_0, T_0\rangle = - \Omega^2$.

\subsection{Total conserved quantities for an asymptotically AdS spacetime}
In this subsection, we evaluate the large sphere limit of the quasi-local conserved quantities for asymptotically AdS initial data sets and show that their limits recover the total conserved quantities for asymptotically AdS initial data sets considered by previous authors. See for example \cite{Abbott-Deser,  Ashtekar-Magnon,CCS, Chrusciel-Herzlich,Chruscie-Nagy, Chrusciel-Maerten-Tod, Gibbons-Hull-Warner, Henneaux-T, Maerten, Wang-Xie-Zhang, Xie-Zhang}.

We first review the AdS spacetime and its Killing fields. Take $\mathbb{R}^{3,2}$ with the coordinate system $(y^0, y^1, y^2, y^3, y^4)$ and the metric \[-(dy^4)^2+\sum_{i=1}^3(dy^i)^2-(dy^0)^2.\]
 $ AdS$ can be identified with the timelike hypersurface
given by \[-(y^4)^2+\sum_{i=1}^3(y^i)^2-(y^0)^2=-1.\] Note that the group $SO(3,2)$ leaves this hypersurface invariant and thus the isometry group of $AdS$ is $SO(3,2)$, which is $10$ dimensional. 

The static chart of  $AdS$ comes from the following parametrization:
\begin{align*}
y^0&=\sqrt{1+r^2}\sin t\\
y^i&=r \tilde x^i\\
y^4&=\sqrt{1+r^2}\cos t.
\end{align*}
We have the following basis for the Killing vector fields: the time translating Killing field
$\frac{\partial}{\partial t}=y^4\frac{\partial}{\partial y^0}- y^0\frac{\partial}{\partial y^4}$, the first set of boost fields \begin{equation}\label{boost_1}\mathfrak{p}^i=y^i\frac{\partial}{\partial y^0}+y^0\frac{\partial}{\partial y^i},\end{equation} the second set of boost fields \begin{equation}\label{boost_2}\mathfrak{c}^i= y^i\frac{\partial}{\partial y^4}+y^4\frac{\partial}{\partial y^i},\end{equation} and the rotation Killing fields $\mathfrak{j}^{k}=\epsilon_{ijk}y^i\frac{\partial}{\partial y^j}$. 

In the static chart of AdS, the metric is of the form
\[  -(1+r^2) dt^2 + \frac{dr^2}{1+r^2} + r^2 (d\theta^2+\sin^2\theta d\phi^2).  \]
The static slice $t=0$ is totally geodesics and the induced metric is the hyperbolic metric. We consider  asymptotically AdS initial data sets as follows:
\begin{definition}\label{a_h_coordinates}
An initial data $(M,g,k)$ is said to be asymptotically AdS if 
there exists a compact subset $K$ of $M$ such that $M\backslash K$ is diffeomorphic to a finite union  of ends $\cup ( \mathbb{H}^3\backslash B_\alpha )$ where each $B_\alpha$ is a geodesic ball in $\mathbb{H}^3$. On each end, under the diffeomorphism, the metric $g$ takes the form: 
\[g= g_{rr} dr^2+ 2 g_{ra}dr du^a + g_{ab} du^adu^b,\]
where
\[\begin{split}
g_{rr} = \frac{1}{r^2} - \frac{1}{r^4}+ \frac{g_{rr}^{(-5)}}{r^5}+O(r^{-6}), \quad  g_{ra}= O(r^{-3}),  \quad
g_{ab} = r^2 \tilde \sigma_{ab}+ \frac{g_{ab}^{(-1)}}{r}+O(r^{-2}),
\end{split}\]
and
\[
k_{rr} = O(r^{-5}) \quad  k_{ra}=k^{(-3)}_{ra} +O(r^{-4}),  \quad k_{ab} =   k^{(-1)}_{ab}  +O(r^{-2}).
\]
\end{definition}
Let $\Sigma_r$ be the coordinate spheres on an end of an asymptotically AdS initial data set. In the following theorem, we evaluate the limit of the quasi-local conserved quantities in terms of the expansion of $g$ and $k$.
\begin{theorem} \label{thm_new}
Let $(M,g,k)$ be an asymptotically AdS initial data set as in Definition \ref{a_h_coordinates} and  $\Sigma_r$ be the coordinate spheres on an end. Let $X_r$ be the isometric embedding of $\Sigma_r$ into the static slice $t=0$ of the AdS spacetime such that $y^i (X_r) = r \tilde x^i +O(1)$ and $\Omega (X_r)= r+O(1)$. We have
\begin{align}
\label{total_energy_p} \lim_{r\to \infty} E(\Sigma_r,X_r,\frac{\partial}{\partial t},\frac{\partial}{\partial t}) =& \frac{1}{8 \pi} \int \left[ g^{(-5)}_{rr} + \frac{3}{2} tr_{S^2} g^{(-1)}_{ab}\right] d S^2 \\ 
\label{total_momentum_p}\lim_{r\to \infty} E(\Sigma_r,X_r,\frac{\partial}{\partial t},\mathfrak{p}^i) =&   \frac{1}{8 \pi} \int  \tilde x^i \left[g^{(-5)}_{rr} + \frac{3}{2} tr_{S^2} g^{(-1)}_{ab}\right]  d S^2\\
\label{total_center_mass_p}  \lim_{r\to \infty} E(\Sigma_r,X_r,\frac{\partial}{\partial t},\mathfrak{c}^i) =&  -\frac{1}{8 \pi} \int    \tilde x^i \tilde \nabla^a  k^{(-3)}_{ra} d S^2\\
\label{total_angular_momentum_p} \lim_{r\to \infty} E(\Sigma_r,X_r,\frac{\partial}{\partial t},\mathfrak{j}^i) =& \frac{1}{8 \pi}  \int   \tilde x^i \Big (\tilde{\epsilon}^{ab}\tilde{\nabla}_b k^{(-3)}_{ra} \Big) d S^2.
\end{align}
\end{theorem}
\begin{proof}
First, we compute the expansion of $(\sigma,|H|,\alpha_H)$ on $\Sigma_r$  in the following lemma.
\begin{lemma}
On $\Sigma_r$, we have the following expansions:
\begin{align}
\label{expansion_metric}\sigma_{ab} =& r^2 \tilde \sigma_{ab}+ \frac{g_{ab}^{(-1)}}{r}+O(r^{-2})\\
\label{expansion_mean_curvature_physical}|H| = & 2+ \frac{1}{r^2} - \frac{ g^{(-5)}_{rr} + \frac{3}{2} tr_{S^2} g^{(-1)}_{ab} }{r^3} + O(r^{-4})\\
\label{expansion_torsion_physical} (\alpha_H)_a=& -\frac{k^{(-3)}_{ra}}{r^2} +O(r^{-3}).
\end{align}
\end{lemma}
\begin{proof}
The computation of $|H|$ is the same as in Lemma 3.1 of \cite{Chen-Wang-Yau4} for an asymptotically hyperbolic initial data set. The only difference is that, in the asymptotically AdS case, $\langle H, e_4 \rangle = O(r^{-3})$ and does not contribute. For $\alpha_H$, we recall from \cite{Wang-Yau3} that
\[  ( \alpha_{H})_a =- k(e_3, \partial_a) + \nabla_a \theta \]
where 
\[  \sinh (\theta) = \frac{- \langle H, e_4 \rangle }{|H|}. \]
The formula follows since the leading term of $e_3$ is $r \frac{\partial}{\partial r}$.
\end{proof}
We are now ready to evaluate the limits of the quasi-local conserved quantities. The static slice $t=0$ corresponds to the hypersurface $y^0=0$. We  have $y^4 = \Omega$, 
\[
\begin{split}
\rho= \frac{H_0-|H|}{\Omega}   \qquad  \textrm{and}  \qquad j = \alpha_H.
\end{split}
\]
For the conserved quantities corresponding to $\frac{\partial}{\partial t}$ and $\mathfrak{p}^i$, we observe that they are normal to the hypersurface and the term in \eqref{qlcq2} that involves the quasi-local momentum density $j$ vanishes.  As a result,
\[    
\begin{split}
\lim_{r\to \infty} E(\Sigma_r,X_r,\frac{\partial}{\partial t},\frac{\partial}{\partial t}) =& \frac{1}{8 \pi}\lim_{r \to \infty} \int \Omega(H_0 -|H|)  d \Sigma_r\\
\lim_{r\to \infty} E(\Sigma_r,X_r,\frac{\partial}{\partial t},\mathfrak{p}^i) =& \frac{1}{8 \pi}\lim_{r \to \infty} \int  y^i (H_0 -|H|)  d \Sigma_r\\
\end{split}
\]
since 
\[
\begin{split}
\langle \frac{\partial}{\partial t}, \frac{\partial}{\partial t} \rangle=  -\Omega^2 \qquad  \textrm{and} \qquad \langle \frac{\partial}{\partial t},  \mathfrak{p}^i \rangle=  \Omega y^i.
\end{split}
\]
From \eqref{expansion_metric} for $\sigma_{ab}$ and the linearized isometric embedding equation, we conclude that
\[  H_0 = 2 + \frac{1}{r^2} + O(r^{-4}). \] \eqref{total_energy_p} and \eqref{total_momentum_p} follow from \eqref{expansion_mean_curvature_physical}.

On the other hand, $\mathfrak{c}^i$ and $\mathfrak{j}^{i}$ are normal to $\frac{\partial}{\partial t}$. As a result, for the conserved quantities corresponding to these vector fields, the term in \eqref{qlcq2} that involves the quasi-local energy density $\rho$ vanishes. Hence, 
\[    
\begin{split}
\lim_{r\to \infty} E(\Sigma_r,X_r,\frac{\partial}{\partial t},\mathfrak{c}^i) =&  -\frac{1}{8 \pi}\lim_{r \to \infty} \int  \alpha_H(  \Omega \nabla y^i - y^i \nabla \Omega)  d \Sigma_r\\
\lim_{r\to \infty} E(\Sigma_r,X_r,\frac{\partial}{\partial t},\mathfrak{j}^i) =& -\frac{1}{8 \pi} \lim_{r \to \infty} \int  \alpha_H( \Omega \epsilon_{kji}y^k \nabla y^j) d \Sigma_r.
\end{split}
\]
\eqref{total_center_mass_p} and \eqref{total_angular_momentum_p} follow from \eqref{expansion_torsion_physical},
\[
\begin{split}
(\frac{\partial}{\partial y^i})^\top =\nabla y^i \qquad \textrm{and} \qquad (\frac{\partial}{\partial y^4})^\top =- \nabla \Omega.
\end{split}
\]
\end{proof}
This leads to the following definition for the total conserved quantities for an asymptotically AdS initial data set.
\begin{definition}\label{conserved}
For an asymptotically AdS initial data set in the sense of Definition \ref{a_h_coordinates}, the 10 total conserved quantities
$E$, $P^i$, $C^i$ and $J^i$ corresponding to $\frac{\partial}{\partial t}$, $\mathfrak{p}^i$, $\mathfrak{c}^i$ and $
\mathfrak{j}^i$, respectively, are defined to be
\begin{align}
\label{total_energy} E =&  \frac{1}{8 \pi} \int \left[ g^{(-5)}_{rr} + \frac{3}{2} tr_{S^2} g^{(-1)}_{ab}\right] d S^2\\
\label{total_momentum}P^i =&  \frac{1}{8 \pi} \int  \tilde x^i (g^{(-5)}_{rr} + \frac{3}{2} tr_{S^2} g^{(-1)}_{ab}) d S^2\\
\label{total_center_mass} C^i=& \frac{1}{8 \pi} \int    \tilde x^i \tilde \nabla^a  k^{(-3)}_{ra} d S^2\\
\label{total_angular_momentum}J^i =& \frac{1}{8 \pi}  \int   \tilde x^i \Big (\tilde{\epsilon}^{ab}\tilde{\nabla}_b k^{(-3)}_{ra} \Big) d S^2.
\end{align}
\end{definition}
Total conserved quantities (or global charges) of asymptotically AdS initial data sets have been studied extensively using the Hamiltonian of asymptotically Killing fields \cite{Abbott-Deser,  Ashtekar-Magnon, Chrusciel-Herzlich,Chruscie-Nagy, Chrusciel-Maerten-Tod, Gibbons-Hull-Warner, Henneaux-T, Maerten, Wang-Xie-Zhang, Xie-Zhang}. We review the construction below and prove that they are the same as the total conserved quantities in Definition \ref{conserved} under the asymptotic assumptions in Definition \ref{a_h_coordinates}.

Let $K$ be a Killing field of  the AdS spacetime.  Let $V$ and $Y$  be the normal component and the  tangential component of $K$ to the static slice of the  AdS spacetime, respectively.  The Hamiltonian $H(V,Y)$ corresponding to $K$ is 
\begin{equation} \label{total_hamiltonian}
H(V,Y) = \lim_{r \to \infty} \frac{1}{8 \pi} \int_{\Sigma_r} [\mathbb U_i(V) +  \mathbb V_i(Y)]  \nu^i d\Sigma_r
\end{equation}
where
\[ 
\begin{split}
\mathbb U_i(V)  = & V g^{jl}\partial_j g_{i l} + D^ j V (g_{ij} - g^{\mathbb H^3}_{ij}), \\
 \mathbb V_i(Y) =& (k_{ij} - tr k \,\,g_{ij})Y^j. 
\end{split}
\]
\begin{proposition}
For asymptotically AdS initial data sets with expansion given in Definition \ref{a_h_coordinates}, the total conserved quantities in Definition \ref{conserved} agree with the total conserved quantities (global charges) in \eqref{total_hamiltonian}.
\end{proposition}
\begin{proof}
In \cite{Miao-Tam-Xie}, Miao, Tam and Xie compute the limit of the Brown-York mass and prove that 
\[
\begin{split}
 \frac{1}{8 \pi}\lim_{r \to \infty} \int r(H_0 - h)  d \Sigma_r  =& H(V(\frac{\partial}{\partial t}), Y(\frac{\partial}{\partial t})), \\
 \frac{1}{8 \pi}\lim_{r \to \infty} \int r\tilde x^i (H_0 - h)  d \Sigma_r  =&  H(V(\mathfrak{p}^i), Y(\mathfrak{p}^i))
\end{split}
\]
where $h$ is the mean curvature of $\Sigma_r$ in $M$. We conclude that
\[
\begin{split}
 E  =& H(V(\frac{\partial}{\partial t}), Y(\frac{\partial}{\partial t})), \\
 P^i  =&  H(V(\mathfrak{p}^i), Y(\mathfrak{p}^i)),
\end{split}
\]
since  $h-|H| = O(r^{-6})$.
 
For $C^i$ and $J^i$, we observe that $\mathfrak{c}^i$ and $\mathfrak{j}^i$ are tangent to the static slice $t=0$ and it is easy to see that
\[
\begin{split}
 C^i  =& H(V(\mathfrak{c}^i), Y(\mathfrak{c}^i))\\
 J^i  =&  H(V(\mathfrak{j}^i), Y(\mathfrak{j}^i)).
\end{split}
\]
\end{proof}

We compare the conserved quantities for  asymptotically AdS initial data sets to  the conserved quantities for asymptotically hyperbolic initial data sets we studied in \cite{Chen-Wang-Yau4}. \eqref{total_energy} and \eqref{total_momentum} resemble the total energy-momentum for the hyperbolic case (see Definition 1.4 of \cite{Chen-Wang-Yau4}). However, the second fundamental form $k$ does not contribute to them in the AdS case. The total angular momentum in  \eqref{total_angular_momentum} is the same as the total angular momentum for  the hyperbolic case (see Theorem 7.3 of \cite{Chen-Wang-Yau4}). The total conserved quantity $C^i$ does not seem to have a good analogy in the hyperbolic case; It is rather different from the center of mass in Theorem 7.3 of \cite{Chen-Wang-Yau4}.
\subsection{Limit of the quasi-local energy}
In this subsection, we evaluate the limit of the quasi-local energy at the infinity of asymptotically AdS initial data sets and show that it converges to the linear function dual to the total conserved quantities. First, we derive an expression for the limit of quasi-local energy $E(\Sigma_r,X_r,T_0)$ for a family of surfaces $\Sigma_r$ and a family of isometric embeddings $X_r$ of $\Sigma_r$ into the reference spacetime. Then we apply the result to the family of coordinate spheres at the infinity of an asymptotically AdS initial data set.
\begin{theorem} \label{linearization_energy}
Let $\Sigma_r$ be a family of surfaces and $X_r$ be a family of isometric embeddings of $\Sigma_r$ into the reference spacetime. Suppose the mean curvature vectors $H$ of $\Sigma_r$ and $H_0$ of $X_r(\Sigma_r)$ are both spacelike for $r > R_0$ and
\[
\lim_{r \to \infty} \frac{|H|}{|H_0|} =1.
\]
Then the limit of $E(\Sigma_r,X_r,T_0)$ is the same as the limit of 
\[
\frac{1}{8 \pi} \int  \left [ - \langle T_0, \frac{J_0}{|H_0|}  \rangle (|H_0| - |H|) + (\alpha_{H_0} - \alpha_{H}) (T_0^\top) \right ] d \Sigma_r
\]
as long as the limits exist. 
\end{theorem}
\begin{proof}
Let $x=\frac{|H|}{|H_0|}$ and 
\[ Y=\frac{ div(T_0^\top) }{|H_0|\sqrt{-  \langle T_0^\perp,T_0^\perp \rangle} }. \]   
In terms of $x$ and $Y$, the quasi-local energy is
\[ 
\begin{split}
   & E(\Sigma_r, X_r,T_0) \\
=& \frac{1}{8 \pi}  
 \int_{\Sigma_r} |H_0|\sqrt{ - \langle T_0^\perp,T_0^\perp \rangle}  \Big [  \sqrt{1 + Y^2}- \sqrt{x^2 + Y^2} - Y \sinh^{-1} Y+Y \sinh^{-1} \frac{Y}{x}  \Big ]  d \Sigma_r \\
&+\frac{1}{8 \pi}   \int_{\Sigma_r}  (\alpha_{H_0} - \alpha_{H})(T_0^\top)  d \Sigma_r.
\end{split}
\]
Let 
\[ f(x) = \sqrt{x^2+Y^2} -Y \sinh^{-1} \frac{Y}{x}. \]
For $x$ close to $1$, we have
\[ f(1) - f(x) = (1-x) \sqrt{1+Y^2} + O((1-x)^2).  \]
We compute
\[ 
\begin{split}
  & \lim_{r \to \infty} \frac{1}{8 \pi}  
   \int_{\Sigma_r} |H_0|\sqrt{ - \langle T_0^\perp,T_0^\perp \rangle}  \Big [  \sqrt{1 + Y^2}- \sqrt{x^2 + Y^2} - Y \sinh^{-1} Y+Y \sinh^{-1} \frac{Y}{x}  \Big ]  d \Sigma_r  \\
=&  \lim_{r \to \infty} \frac{1}{8 \pi}  
   \int_{\Sigma_r} |H_0|\sqrt{ - \langle T_0^\perp,T_0^\perp \rangle} (1- x ) \sqrt{1+Y^2}   d \Sigma_r  \\
= & \lim_{r \to \infty}  \int_{\Sigma_r} ( |H_0| - |H|) \sqrt{ - \langle T_0^\perp,T_0^\perp \rangle}  \sqrt{1+\frac{ div(T_0^\top) ^2}{ -  \langle T_0^\perp,T_0^\perp \rangle |H_0|^2 }}   d \Sigma_r 
\end{split}
\]
Recall that 
\[   T_0 ^\perp =  \sqrt{ - \langle T_0^\perp,T_0^\perp \rangle} \breve e_4 \]
and 
\[
\begin{split}
- \langle T_0, \frac{J_0}{|H_0|} \rangle = & - \sqrt{ - \langle T_0^\perp,T_0^\perp \rangle}   \langle \breve e_4, \frac{J_0}{|H_0|} \rangle \\
=&  \sqrt{ - \langle T_0^\perp,T_0^\perp \rangle}  \sqrt{1+\frac{ div(T_0^\top) ^2}{ -  \langle T_0^\perp,T_0^\perp \rangle |H_0|^2 }},
\end{split}
\]
where  \eqref{gauge_angle} and \eqref{gauge} are used in the last equality. This finishes the proof of the theorem.
\end{proof}
We are ready to show that for an asymptotically AdS initial data set, the large sphere limit of the quasi-local energy is the linear function dual to the total conserved quantities in Definition \ref{conserved}. 
\begin{theorem}
Let $(M,g,k)$ be an asymptotically AdS initial data set and $\Sigma_r$  be the coordinate spheres. Let $X_r$ be an isometric embedding of $\Sigma_r$ into the static slice $t=0$ of the AdS spacetime such that $y^i (X_r) = r \tilde x^i +O(1)$ and $\Omega (X_r)= r+O(1)$. Consider the observer 
\[  T_0 = A (y^0\frac{\partial}{\partial y^4}-y^4\frac{\partial}{\partial y^0})+B_k(y^k\frac{\partial}{\partial y^0}+y^0\frac{\partial}{\partial y^k}) + D_k( y^k\frac{\partial}{\partial y^4}+y^4\frac{\partial}{\partial y^k}) +F_k\epsilon_{ijk}y^i\frac{\partial}{\partial y^j}.\]
We have
\[
\lim_{r \to \infty} E(\Sigma_r ,X_r,T_0) = AE + B_k P^k +D_k C^k + F_k J^k.
\]
where $E$, $P^k$, $C^k$ and $J^k$ are the total conserved quantities.
\end{theorem}
\begin{proof}
Recall that the static slice $t=0$ is the same as the hypersurface $y^0=0$.
The Killing fields $ y^0\frac{\partial}{\partial y^4}-y^4\frac{\partial}{\partial y^0}$ and $y^k\frac{\partial}{\partial y^0}+y^0\frac{\partial}{\partial y^k}$ are normal to the hypersurface. On the other hand, $\epsilon_{ijk}y^i\frac{\partial}{\partial y^j}$ and $y^k\frac{\partial}{\partial y^4}+y^4\frac{\partial}{\partial y^k}$ are tangent to the hypersurface. As a result, 
\[  
- \langle T_0, \frac{J_0}{|H_0|} \rangle  =  - \langle T^{\perp}_0, \frac{J_0}{|H_0|} \rangle= r(A + B_i \tilde x^i) + O(1).
\]
It is also easy to verify that 
\[ (T_0^\top)_a =  r^2(  D_k \tilde x^k_a +  F_k \epsilon_{ijk}\tilde x ^i \tilde x ^j_a   ) + O(r). \]
The theorem follows directly from Theorem \ref{thm_new} and Theorem \ref{linearization_energy}.
\end{proof}

\subsection{Evolution of the total conserved quantities under the Einstein equation}
In this subsection, we study the evolution of  the total conserved quantities for asymptotically AdS initial data sets under the Einstein equation.  

We assume that the initial data set $(M,g,k)$ satisfies the vacuum constraint equation (with cosmological constant $\kappa=-1$)
\begin{equation}\label{constraint}
\begin{split}
R(g) + (tr_g k)^2 - |k|^2 = & -6 \\
\bar \nabla ^{i}k_{ij}  - \partial_j (tr_g k) =& 0
\end{split}
\end{equation}
where $\bar \nabla$ is the covariant derivative with respect to $g$.

We shall fix an asymptotically flat coordinate system on $M$ with respect to $(g_{ij}(0), k_{ij}(0))$ and consider a family $(g_{ij}(t), k_{ij}(t))$ that evolves according to the vacuum Einstein  evolution equation (with cosmological constant $\kappa=-1$)
\begin{equation} \label{evolution}
\begin{split}
\partial _t g_{ij} & =-2N k_{ij}+(\mathcal L _{\gamma} g)_{ij}\\
\partial _t k_{ij} & = -\bar \nabla_i\bar \nabla_j N+N\left(R_{ij} +g_{ij}+ (tr k) k_{ij} - 2 k_{il} k^l \,_j\right)  + (\mathcal L _{\gamma} k)_{ij}
\end{split}
\end{equation}
where $N$ is the lapse function, $\gamma$ is the shift vector, and $\mathcal L$ is the Lie derivative.

\begin{theorem}\label{total_AdS_evolve}
Let $(M,g,k)$ be an  asymptotically AdS initial data set.  Let $(M,g(t),k(t))$ be the solution to the vacuum Einstein equation with $g(0)=g$ and $k(0)=k$, and with lapse $N=\sqrt{r^2+1}$ and a vanishing shift vector. Let $E(t)$, $P^i(t)$, $C^i(t)$ and $J^i (t)$ be the total conserved quantities for $(M,g(t),k(t))$ defined in Definition \ref{conserved}. We have
\[
\begin{split}
\partial_t E(t) = & 0 \\
\partial_t P^i(t) = &  -C^i(t) \\
\partial_t C^i(t) = &P^i(t) \\
\partial_t J^i(t) = & 0.
\end{split}
\]
\end{theorem}
\begin{remark}
The evolution equations for $E$ and $P^i$ are proved previously in \cite[Theorem 5.1]{CCS} with a different convention for $C^i$. 
\end{remark}
\begin{proof}
Let $h(t)$ be the mean curvature of $\Sigma_r$ in the hypersurface  $(M,g(t))$ and $H(t)$ be the mean curvature vector of $\Sigma_r$ in the spacetime. We have  $|H(t)| = h(t) + O(r^{-4})$ and the following formula for $h(t)$ (see for example \cite{Wang-Yau3, Chen-Wang-Yau3}) :
\[ h(t) = \frac{\partial_r \ln \sqrt{det(\sigma_{ab})} - \nabla ^a g_{ra} }{\sqrt{g_{rr} - \sigma^{ab}g_{ra}g_{rb}}}  \]
where $g_{ra}$ is viewed as a 1-form on $\Sigma_r$ and $\nabla$ is the covariant derivative with respect to the induced metric on $\Sigma_r$.

From the Einstein equation  \eqref{evolution}, we derive
\[  \partial_t g_{ab} = - 2\sqrt{r^2+1}k_{ab}. \]
As a result,
\[
\begin{split}
\partial_t |H| = & \partial_t h(t) + O(r^{-4})\\
=&  -2 r^3 k_{rr} - r \sigma^{ab}k_{ab} + r ^2\sigma^{ab} \partial_r k_{ab} -2r^2 \nabla^a k_{ar} +O(r^{-4}).
\end{split}
\]
From the vacuum constraint equation, \eqref{constraint} we derive
\begin{equation}\label{momentum_constraint_r}
  g^{ij}\bar \nabla_{i} k_{jr} = \partial_r (tr_g k).  \end{equation}
The left hand side of \eqref{momentum_constraint_r} is 
\[
\begin{split}
g^{ij}\bar \nabla_{i} k_{jr}
= & g^{rr}\bar \nabla_{r} k_{rr} + g^{ab}\bar \nabla_{a} k_{br} + O(r^{-6})\\
= & r^2 \partial_r k_{rr} + 4r k_{rr} + \nabla^a k_{ar}  - \frac{1}{r} \sigma^{ab}k_{ab} + O(r^{-6})
\end{split}
\]
On the other hand, the right hand side of \eqref{momentum_constraint_r}  is 
\[
\partial_r (tr_g k) = r^2\partial_r k_{rr} + 2r k_{rr} + \sigma^{ab} \partial_r k_{ab} - 2r  \sigma^{ab}k_{ab} + O(r^{-6}).
\]
As a result, \eqref{momentum_constraint_r} implies
\[ \nabla^a k_{ar} = - 2r k_{rr} + \sigma^{ab} \partial_r k_{ab} - r  \sigma^{ab}k_{ab} + O(r^{-6}) \]
and
\[  \partial_t h(t) =- r^2 \nabla^a k_{ar} + O(r^{-4}).\]
This proves the evolution equations for $E$ and $P^i$.

To evaluate $\partial_t C^i$ and $\partial_t J^i$, we start with the evolution equation of the second fundamental  form \eqref{evolution}, which implies
\[  \partial_t k_{ar} = r Ric_{ar} + O(r^{-4}). \]
Let $\AA $ be the traceless part of second fundamental form of the surface $\Sigma_r$ in the hypersurface $(M,g(t))$. The Codazzi equation reads 
\[  \nabla^a \AA_{ab} - \frac{1}{2} \nabla_a h = Ric_{ar} . \]
The evolution of $C^i$ follows from taking the divergence of the above equation, multiplying with $r^2 x^i$, and integrating over $\Sigma_r$. For the evolution of $J^i$, we take the curl of the above equation instead.
\end{proof}
From the above theorem, it follows that the rest mass of asymptotically initial data defined by the authors in \cite{Chen-Hung-Wang-Yau} is invariant under the Einstein equation.
\begin{corollary}
Let $(M,g,k)$ be an  asymptotically AdS initial data set.  Let $(M,g(t),k(t))$ be the solution to the vacuum Einstein equation with $g(0)=g$ and $k(0)=k$, and with lapse $N=\sqrt{r^2+1}$ and a vanishing shift vector. Let $m(t)$ be the rest mass of the data $(M,g(t),k(t))$. Then we have
\[ \partial_t m(t)=0.\]
\end{corollary}
\begin{proof}
Let $\vec{p}=(P^1,P^2,P^3)$, $\vec{c}=(C^1,C^2,C^3)$ and $\vec{j}=(J^1,J^2,J^3)$. From Theorem 6.7 of \cite{Chen-Hung-Wang-Yau}, the rest mass $m$ in terms of the total conserved quantities is
\[  m^2=\frac{1}{2}(\alpha+\sqrt{\beta}) \]
where
\[
\begin{aligned}
\alpha&=E^2+|\vec{j}|^2-|\vec{p}|^2-|\vec{c}|^2\\
\beta&=(E^2-|\vec{j}|^2-|\vec{p}|^2-|\vec{c}|^2)^2-4|\vec{j}\times\vec{p}|^2-4|\vec{p}\times\vec{c}|^2-4|\vec{c}\times\vec{j}|^2+8E\vec{c}\cdot (\vec{p}\times\vec{j}).
\end{aligned}
\]
The corollary follows from Theorem \ref{total_AdS_evolve} by a direct computation.
\end{proof}
 
\end{document}